\documentclass[english,10pt]{article}
\usepackage{geometry}
\usepackage{float}
\usepackage{mathrsfs,bm}
\usepackage{amsmath}
\usepackage{amssymb}
\usepackage{fancyhdr}
\usepackage{epsfig}
\usepackage{subfig}
\usepackage{amsthm}
\usepackage{mathrsfs}
\usepackage{amsfonts}
\usepackage{mathrsfs}
\usepackage{enumerate}

\makeatletter

\floatstyle{ruled}
\newfloat{algorithm}{tbp}{loa}
\providecommand{\algorithmname}{Algorithm}
\floatname{algorithm}{\protect\algorithmname}

\@ifundefined{definecolor}{\@ifundefined{definecolor}
 {\@ifundefined{definecolor}
 {\usepackage{color}}{}
}{}
}{}

\usepackage[auth-sc,affil-it]{authblk}
\newcommand{\bbR}{\mathbb R}

\newtheorem{theorem}{Theorem}[section]
\newtheorem{lem}{Lemma}[section]
\newtheorem{rem}{Remark}[section]
\newtheorem{prop}{Proposition}[section]

\newcounter{hypA}
\newenvironment{hypA}{\refstepcounter{hypA}\begin{itemize}
  \item[({\bf A\arabic{hypA}})]}{\end{itemize}}

\newcommand{\EE}{\mathbb{E}}

\newcommand{\QQ}{\mathbb{Q}}
\newcommand{\var}{\mathrm{Var}}

\renewcommand{\phi}{\varphi}

\usepackage{babel}

\begin{document}

\begin{center}

{\Large \textbf{On the Behaviour of the Backward Interpretation of Feynman-Kac Formulae under Verifiable Conditions}}

\bigskip

BY AJAY JASRA 

{\footnotesize Department of Statistics \& Applied Probability,
National University of Singapore, Singapore, 117546, SG.}\\
{\footnotesize E-Mail:\,}\texttt{\emph{\footnotesize staja@nus.edu.sg}}
\end{center}

\begin{abstract}
In the following article we consider the time-stability associated to the sequential Monte Carlo (SMC) estimate of the backward interpretation of Feynman-Kac Formulae.
This is particularly of interest in the context of performing smoothing for hidden Markov models (HMMs). We prove a central limit theorem (CLT) under weaker assumptions 
than adopted in the literature. We then show that the associated asymptotic variance expression, for additive functionals grows at  most linearly in time, under hypotheses
that are weaker than those currently existing in the literature. The assumptions are verified for some state-space models.
\end{abstract}

\section{Introduction} 

Feynman-Kac formulae provide a very general description of several models, such as hidden Markov models (see e.g.~\cite{cappe}), used in statistics, physics, computational biology and many more; see \cite{delmoral}.
For a measurable space $(\mathsf{X},\mathcal{B}(\mathsf{X}))$, $f:\mathsf{X}\rightarrow\mathbb{R}$ (bounded for now), the  Feynman-Kac formula associated to the  $n$-time marginal, $n\geq 1$ is:
$$
\eta_n(f) := \frac{\gamma_n(f)}{\gamma_n(1)}
$$
with, for $\mu$ a probability measure on $\mathsf{X}$, $G_n:\mathsf{X}\rightarrow\mathbb{R}_+$ (bounded), $n\geq 0$, $M_n:\mathsf{X}\times\mathcal{B}(\mathsf{X})\rightarrow[0,1]$, $n\geq1$
\begin{equation}
\gamma_n(f) := \int_{\mathsf{X}^{n+1}}f(x_{n}) \Big[\prod_{p=0}^{n-1} G_p(x_p)\Big] \mu(dx_0) \prod_{p=1}^n M_{p}(x_{p-1},dx_p)\label{eq:gamma_def}.
\end{equation}
We take $\eta_0=\mu$. In the context of HMMs, $\eta_n$ represents the predictor, equivalently, the conditional distribution of the signal given the observations up-to time $n-1$.
In many practical applications, such as the smoothing problem in HMMs, one is interested in the formula, for $F_n:\mathsf{X}^{n+1}\rightarrow\mathbb{R}$ (bounded for now), 
$$
\mathbb{Q}_n(F_n) = \frac{\int_{\mathsf{X}^{n+1}}F_n(x_{0},\dots,x_n) \Big[\prod_{p=0}^{n-1} G_p(x_p)\Big] \mu(dx_0) \prod_{p=1}^n M_{p}(x_{p-1},dx_p)}{\int_{\mathsf{X}^{n+1}} \Big[\prod_{p=0}^{n-1} G_p(x_p)\Big] \mu(dx_0) \prod_{p=1}^n M_{p}(x_{p-1},dx_p)}.
$$
In practice this formula, as well as that for the predictor is unavailable analytically and one must resort to numerical approximation procedures, in order to compute it.
We remark that $\mathbb{Q}_n(F_n)$ is of interest, not only for smoothing for HMMs, but many other application areas; see for instance \cite{dds1} and the references therein.
In this article we focus on the numerical approximation of $\mathbb{Q}_n(F_n)$ and simultaneously $\eta_n(f)$. The latter task is often done quite well using SMC methods, as we now discuss.

SMC methods are designed to approximate a sequence of probability distributions of increasing dimension. The method uses $N\geq 1$
samples (or particles) that are generated in parallel, and are propagated via importance sampling (i.e.~via Markov proposals and importance weights) and resampling methods. The approach can provide estimates of expectations with respect to this sequence of distributions of increasing accuracy as $N$ grows. Standard SMC methodology is by now very well understood with regards to its convergence properties and several consistency results have been proved (see  e.g.~\cite{delmoral,douc}) along with the stability in time of the error of the algorithm \cite{douc2,whiteley} in the context of filtering for HMMs. These latter results are particularly important
as due to the sequential in time nature of the inference; one does not want the errors over time to accumulate.  

As noted above, SMC can be very useful for approximating $\eta_n(f)$. However, it is well known due to the path degeneracy problem (see \cite{doucet}) that the standard SMC approach, of cost $\mathcal{O}(N)$ per time step, for approximating $\mathbb{Q}_n(F_n)$ performs very badly.
For example, consider the CLT for the standard SMC approximation of $\mathbb{Q}_n(F_n)$, call it $\mathbb{Q}^{N,S}_n(F_n)$ with $F_n(x_0,\dots,x_n)=\sum_{p=0}^n f_p(x_p)$, $f_p:\mathsf{X}\rightarrow\mathbb{R}$, (additive functionals - this is of particular interest in application areas):
$$
\sqrt{N}[\mathbb{Q}^{N,S}_n(F_n)-\mathbb{Q}_n(F_n)] \Rightarrow\mathcal{N}(0,\sigma_n^{2,S}(F_n))
$$
where $\Rightarrow$ denotes convergence in distribution as $N\rightarrow+\infty$ and $\mathcal{N}(0,\sigma_n^{2,S}(F_n))$ is a one-dimensional Gaussian distribution with zero mean and variance $\sigma^{2,S}_n(F_n)$.
\cite{poya} show that, under strong assumptions, $\sigma^{2,S}_n(F_n)\geq c(n)$, with $c(n)$, $\mathcal{O}(n^2)$, i.e.~grows quadratically in the time parameter.

One SMC approach designed to deal with these afore-mentioned issues is that of the forward filtering backward smoothing algorithm (FFBS) of  \cite{doucet1,godsill} and in particular
the SMC approximation of the backward interpretation of Feynman-Kac formulae, write this $\mathbb{Q}^{N}_n(F_n)$. This is a `forward only' approximation of the FFBS algorithm, which is of cost $\mathcal{O}(N^2)$ per time step, 
and several convergence results for this algorithm (and FFBS), including a CLT are proved in  \cite{dds1,douc,dubarry}; the assumptions used are fairly strong and do not always apply on non-compact state-spaces $\mathsf{X}$. The $\mathcal{O}(N^2)$ cost per time step is counter-balanced by  the time-behaviour of (an appropriateley defined) error in approximating $\mathbb{Q}_{n}(F_n)$ for $F_n$ additive; it can be
no worse than linear in time (see e.g.~\cite{dubarry}), versus the $\mathcal{O}(n^2)$ for standard SMC.  For instance, \cite{dds1} show that for $F_n$ additive, as $\sqrt{N}[\mathbb{Q}^{N}_n(F_n)-\mathbb{Q}_n(F_n)] \Rightarrow\mathcal{N}(0,\sigma^{2}_n(F_n))$, under some strong hypotheses:
$$
\sigma^{2}_n(F_n) \leq c(n+1)
$$
with $c<+\infty$ not depending upon $n$.
As already remarked, these theoretical results are derived under strong assumptions: In this work we weaken the hypotheses used in previous articles (such as \cite{dds1,douc,dubarry}). 
A related idea, the forward filtering backward simulation algorithm in \cite{douc} has cost $\mathcal{O}(N)$ but we do not consider it in this article.

In the analysis of SMC algorithms, time-stability is often posed as follows. 
Writing $\eta_n^N(f)$ as the SMC approximation of $\eta_n(f)$, one has under minimal assumptions that $\sqrt{N}[\eta_n^N(f)-\eta_n(f)]\Rightarrow\mathcal{N}(0,\vartheta^{2}_n(f))$ and in the literature an often proved result, under additional assumptions,
is that
$$
\vartheta^{2}_n(f) \leq c
$$
where $c$ does not depend upon $n$.
The time stability of SMC has been studied in many papers (e.g.~\cite{delmoral_stab,heine}), but, only recently have assumptions
been weakened, for example in \cite{douc2,vanhandel,whiteley}. 
The assumptions used in the early work of \cite{delmoral_stab} relied on very strong mixing assumptions associated to the underlying Markov chain of the Feynman-Kac formula. Significant efforts were made to weaken this assumption
and recent work of \cite{douc2,whiteley} (see also \cite{whiteley1}). These works, in the context of the asymptotic variance in the CLT associated to the SMC approximation of the $n-$time Feynman-Kac marginal, has used local Doeblin (see \cite{douc3}) and multiplicative drift condions (see \cite{kont})
to provide more verifiable assumptions for the stability of SMC.
We use similar assumptions to \cite{whiteley} to weaken the assumptions used in \cite{dds1,douc1} for: 
\begin{enumerate}
\item{Proving a CLT for the SMC approximation of  the backward interpretation of Feynman-Kac formulae (Theorem \ref{theo:clt}), that is
$$
\sqrt{N}[\mathbb{Q}^{N}_n(F_n)-\mathbb{Q}_n(F_n)] \Rightarrow\mathcal{N}(0,\sigma^{2}_n(F_n)).
$$
}
\item{Giving a linear-in-time bound on the associated asymptotic variance expression when the function is additive (Theorem \ref{theo:asymp_var_bound}), that is, for $F_n(x_0,\dots,x_n)=\sum_{p=0}^n f_p(x_p)$
$$
\sigma^{2}_n(F_n) \leq c(n+1)
$$
where $c$ does not depend upon $n$.} 
\end{enumerate}

This article is structured as follows. In Section \ref{sec:prelim} we give our notations, the algorithm and estimates along with our assumptions. In Section \ref{sec:clt} the CLT is proved. 
In Section \ref{sec:av_cont} we prove the linear in time increase of the asymptotic variance expression for additive functions. In Section \ref{sec:verify} we give an example of an HMM were
our assumptions hold.
The appendix contains technical results for the proofs of the CLT and asymptotic variance and is split into two Sections.

\section{Preliminaries}\label{sec:prelim}

\subsection{Notations}

For a kernel $M:\mathsf{X}\times\mathcal{B}(\mathsf{X})\rightarrow\mathbb{R}_+$ and $\sigma-$finite measure $\mu$ on $(\mathsf{X},\mathcal{B}(\mathsf{X}))$ $\mu M(\cdot):=\int_{\mathsf{X}} \mu(dx) M(x,\cdot)$.
For a function $\varphi:\mathsf{X}\rightarrow\mathbb{R}$ and kernel $M$ (resp.~signed measure $\mu$), $M(\varphi)(x):=\int_{\mathsf{X}}\varphi(y)M(x,dy)$ (resp.~$\mu(\varphi):=\int \varphi(y)\mu(dy)$).
For a given function $V:\mathsf{X}\mapsto[1,\infty)$ we denote by $\mathscr{L}_V$
the class of functions $\varphi:\mathsf{X}\rightarrow\bbR$ for which
$$
\|\varphi\|_V := \sup_{x\in \mathsf{X}}\frac{|\varphi(x)|}{V(x)} < +\infty\ .
$$
When $V\equiv 1$ we write $\|\varphi\|_{\infty} := \sup_{x\in \mathsf{X}}|\varphi(x)|$.
We also denote, for a probability measure $\mu$, $\|\mu\|_{V}:=\sup_{|\varphi|\leq V}|\mu(\varphi)|$. The probability measures on $\mathsf{X}$ are denoted $\mathcal{P}$. For $\mu\in\mathcal{P}$ such that $\mu(V)<+\infty$
we denote $\mu\in\mathcal{P}_V$. 
Throughout $c$ is used to denote a constant whose meaning may change, depending upon the context; any (important) dependencies are written as $c(\cdot)$. 
The bounded, real-valued and measurable functions on a space $\mathsf{Z}$ are written $\mathbb{B}_b(\mathsf{Z})$. The notation $x_{k:n}=(x_k,\dots,x_n)$ is used,
with $k<n$.

Recall \eqref{eq:gamma_def} which is defined in terms of potentials $G_n$ and Markov kernels $M_n$.
Throughout the article it is assumed, for a $\sigma-$finite measure $\lambda$ on $\mathsf{X}$ (typically Lebesgue) and each $n\geq 1$:
$$
M_{n}(x_{n-1},dx_n) = H_{n}(x_{n-1},x_n)\lambda(dx_n)
$$
where $H_n:\mathsf{X}^2\rightarrow\mathbb{R}_+$, with $\int_{\mathsf{X}}H_{n}(x_{n-1},x_n)\lambda(dx_n)=1~\forall x_{n-1}\in\mathsf{X}$. We also introduce the semi-group for $n\geq 1$:
$$
Q_{n}(x_{n-1},dx_n) := G_{n-1}(x_{n-1})M_{n}(x_{n-1},dx_n)
$$
with, for $0\leq p\leq n$, $f:\mathsf{X}\rightarrow\mathbb{R}$, $Q_{p,n}(f)(x) := \int f(x_n)  \prod_{q=p+1}^n Q_{q}(x_{q-1},dx_q)$ with the convention $Q_{p,p}=Id$, the identity operator.
We use this semi-group notation for operators that are introduced later on. We will write weak convergence (as $N$ the number of samples grows) as $\Rightarrow$
and convergence in probability as $\rightarrow_{\mathbb{P}}$. We write the $d-$dimensional Gaussian distribution, with mean vector $\mu$ and covariance matrix
$\Sigma$ as $\mathcal{N}_d(\mu,\Sigma)$ and if $d=1$ we drop subscript $d$.

\subsection{Algorithm and Estimate}

The SMC algorithm samples from the joint law
\begin{equation*}
\mathbb{P}\big(\,d(x_0^{1:N},x_1^{1:N},\ldots,x_n^{1:N})\,\big) = \Big(\prod_{i=1}^N \eta_0(dx_0^i)\Big)\prod_{p=1}^n \prod_{i=1}^N\Phi_{p}(\eta_{p-1}^N)(dx_p^i)\ ,
\end{equation*}
where $x_q^{1:N} = (x_q^1,\dots,x_q^N)\in\mathsf{X}^N$ ($0\leq q\leq n$), $\eta_n^{N}$ is the empirical measure 
$\frac{1}{N}\sum_{i=1}^N \delta_{x_n^i}$
and the operator $\Phi_n: \mathcal{P} \to \mathcal{P}$ maps a probability distribution $\mu\in\mathcal{P}$ to the probability measure $\Phi_n(\mu) \in \mathcal{P}$ defined by
\begin{equation*}
\Phi_n(\mu)(dy) = \frac{\mu(G_{n-1}M_{n})(dy)}{\mu(G_{n-1})} \ .
\end{equation*}
The estimate of $\gamma_n(f)$ is $\gamma_n^N(f)=[\prod_{q=0}^{n-1}\eta_q^N(G_q)] \eta_{n}^N(f)$.
Various results have been proved about the convergence associated to $\eta_n^N(\cdot)$ (resp.~$\gamma_n^N(\cdot)$) to $\eta_n(\cdot)$ (resp.~$\gamma_n(\cdot)$); see for instance \cite{delmoral}.

Let $F_n:\mathsf{X}^{n+1}\rightarrow\mathbb{R}$, we will study the SMC approximation of
$$
\mathbb{Q}_n(F_n) = \frac{\int_{\mathsf{X}^{n+1}}F_n(x_{0:n}) \Big[\prod_{p=0}^{n-1} G_p(x_p)\Big] \mu(dx_0) \prod_{p=1}^n M_{p}(x_{p-1},dx_p)}{\int_{\mathsf{X}^{n+1}} \Big[\prod_{p=0}^{n-1} G_p(x_p)\Big] \mu(dx_0) \prod_{p=1}^n M_{p}(x_{p-1},dx_p)}.
$$
Now the backward interpretation (see e.g.~\cite{dds1}) is
$$
\mathbb{Q}_n(F_n)  = \int_{\mathsf{X}^{n+1}} F_n(x_{0:n}) \eta_n(dx_n) \mathcal{M}_n(x_n,dx_{0:n-1})
$$
where
\begin{eqnarray}
\mathcal{M}_n(x_n,dx_{0:n-1}) & = & \prod_{q=1}^n M_{q,\eta_{q-1}}(x_q,dx_{q-1})\label{eq:m_back_path}\\
M_{q,\eta_{q-1}}(x_q,dx_{q-1}) & = & \frac{G_{q-1}(x_{q-1})H_q(x_{q-1},x_q)\eta_{q-1}(dx_{q-1})}{\eta_{q-1}(G_{q-1}H_q(\cdot,x_q))}\nonumber
\end{eqnarray}
we write $\mathcal{M}_n^N$ in \eqref{eq:m_back_path}, when each $\eta_0,\dots,\eta_{n-1}$ are replaced by the empirical versions.
The SMC approximation of  $\mathbb{Q}_n(\cdot)$, written $\mathbb{Q}_n^N(\cdot)$ is
$$
\mathbb{Q}_n^N(dx_{0:n})  =  \eta_n^N(dx_n) \prod_{q=1}^{n} M_{q,\eta_{q-1}^N}(x_q,dx_{q-1})
$$
where the empirical measures $\eta_{q-1}^N$ are defined above.
If  $F_n(x_{0:n})=\sum_{p=0}^nf_p(x_p)$, $f_p:\mathsf{X}\rightarrow\mathbb{R}$, then setting $F_0^N=f_0$, then the $\mathcal{O}(N^2)$ approximation is
$$
\mathbb{Q}_n^N(F_n) = \eta_n^N(F_n^N)
$$
where
$$
F_n^N(x) = f_n(x) + \sum_{i=1}^N \frac{G_{n-1}(x_{n-1}^i) H_n(x_{n-1}^i,x)}{\sum_{j=1}^N G_{n-1}(x_{n-1}^j) H_n(x_{n-1}^j,x)} F_{n-1}^N(x_{n-1}^i).
$$
This is particularly useful for the smoothing problem associated to HMMs.

\subsection{Assumptions}

We make the following hypotheses. (A1-2), (A4-6) are (H1-5) in \cite{whiteley}, except slightly modified to the density notations which naturally occur in many application areas.
(A\ref{hyp:6}) appears to be needed under our analysis, but can be verified in practice. It is not dissimilar to part of (H1) in \cite{delm_g} and, under the other assumptions of this article
could be verified if
$$
H_{n}\in\mathscr{L}_{\overline{v}^{\beta_1}} \quad \textrm{and} \quad \Big(\inf_{x\in C_d} G_{n-1}(x) H_n(x,y)\Big)^{-1}\in\mathscr{L}_{v^{\beta_2}}
$$
with $\beta_1,\beta_2>0$ and $\alpha=\beta_1+\beta_2$, $\alpha$ as in (A\ref{hyp:6}).
A discussion of the assumptions and comparison to \cite{douc3} can be found in \cite{whiteley}. The assumptions are, in general, weaker than those used in \cite{dds1,douc,dubarry} and can be verified on non-compact state-spaces.

\begin{hypA}\label{hyp:1}
There exists a $V:\mathsf{X}\rightarrow[1,\infty)$ unbounded and constants $\delta\in(0,1)$ and $\underline{d}\geq 1$ with the following properties.
For each $d\in(\underline{d},+\infty)$ there exists a $b_d<+\infty$ such that $\forall x\in\mathsf{X}$
$$
\sup_{n\geq 1} Q_n(e^V)(x) \leq e^{(1-\delta)V(x) + b_d\mathbb{I}_{C_d}(x)}
$$
where $C_d=\{x\in\mathsf{X}:V(x)\leq d\}$.
\end{hypA}
\begin{hypA}\label{hyp:3}
$\mu\in\mathcal{P}_v$, with $v=e^{V}$.
\end{hypA}
\begin{hypA}\label{hyp:6}
For every $\alpha\in(0,1/2)$:
$$
\sup_{n\geq 1} \frac{G_{n-1}(x)H_n(x,y)}{\eta_{n-1}(G_{n-1}H_n(\cdot,y))} \in \mathscr{L}_{\overline{v}^{\alpha}}
$$
with $\overline{v}(x,y)^{\alpha} = v(x)^{\alpha}v(y)^{\alpha}$.
\end{hypA}
\begin{hypA}\label{hyp:2}
With $\underline{d}$ as in (A\ref{hyp:1}), for each $d\in[\underline{d},\infty)$
$$
G_{n-1}(x)H_n(x,y) > 0 \quad \forall x,y\in\mathsf{X}, n\geq 1
$$
with $0<\int_{C_d}\lambda(dy)<+\infty$
and there exist $\tilde{\epsilon}_d^{-}>0$ such that
$$
\inf_{n\geq 1} G_{n-1}(x)H_n(x,y) \geq \tilde{\epsilon}_d^{-}, \quad \forall x,y\in C_d.
$$
In addition $\nu_d(dy):=\lambda(dy)\mathbb{I}_{C_d}(y)/\int_{C_d}\lambda(dy)\in\mathcal{P}_v$.
\end{hypA}
\begin{hypA}\label{hyp:4}
With $\underline{d}$ as in (A\ref{hyp:1}),  and $\tilde{\epsilon}_d^-$ as in (A\ref{hyp:2}), for each $d\in[\underline{d},\infty)$ there exist $\tilde{\epsilon}_d^+\in[\tilde{\epsilon}_d^-,\infty)$ such that
$$
\sup_{n\geq 1} G_{n-1}(x)H_n(x,y) \leq \tilde{\epsilon}_d^+, \quad  \forall x,y\in C_d
$$
\end{hypA}
\begin{hypA}\label{hyp:5}
$\sup_{n\geq 0} \sup_{x\in\mathsf{X}} G_n(x) < +\infty$.
\end{hypA}

\section{Central Limit Theorem}\label{sec:clt}

The asymptotic variance in the CLT for the forward-only smoothing (resp.~FFBS) is, under some conditions, \cite[Theorem 3.1]{dds1} (see also \cite{douc}):
$$
\sigma^2_n(F_n) := \sum_{p=0}^n \eta_p\bigg(\bigg[h_{p,n}\Big\{P_{p,n}(F_n)-\frac{\eta_p(D_{p,n}(F_n))}{\eta_p(D_{p,n}(1))}\Big\}\bigg]^2\bigg) 
$$
for the predictor. The operators are, for $0\leq p\leq n$
\begin{eqnarray}
h_{p,n}(x_p) & = & \frac{Q_{p,n}(1)(x_p)}{\eta_p(Q_{p,n}(1))}\nonumber\\
P_{p,n}(F_n)(x_p) & = & \frac{D_{p,n}(F_n)(x_p)}{D_{p,n}(1)(x_p)}\nonumber\\
D_{p,n}(F_n)(x_p) & = & \int \mathcal{M}_p(x_p,dx_{0:p-1})\mathcal{Q}_{p,n}(x_p,dx_{p+1:n})F_n(x_{0:n})\nonumber\\
\mathcal{Q}_{p,n}(x_p,dx_{p+1:n}) & = & \prod_{q=p}^{n-1} Q_{q+1}(x_q,dx_{q+1})\label{eq:q_path_def}).
\end{eqnarray}
With the conventions $D_{0,n}=\mathcal{Q}_{0,n}$ and $D_{n,n}=\mathcal{M}_n$.
We give the CLT under weaker assumptions than considered by \cite{dds1,douc}, but only for bounded functions; we note that (A\ref{hyp:1}) and (A\ref{hyp:6}) need not be time-uniform, but to connect with the next Section, we make them time-uniform. 
Indeed, one can pose (A\ref{hyp:1}) as $Q_n(v)\leq c(n) v^{1-\delta}$.
We suppose that for any $n\geq 0$, $\|G_n\|_{\infty}<+\infty$, below. 

\begin{theorem}\label{theo:clt}
Assume (A\ref{hyp:1}-\ref{hyp:6}).  Suppose that for each $n\geq 0$, $1/G_{n}\in\mathscr{L}_{v^{\delta/2}}$, with $\delta$ as in (A\ref{hyp:1}), then for any $n\geq 0$, $F_n\in\mathbb{B}_b(\mathsf{X}^{n+1})$
$$
\sqrt{N} [\mathbb{Q}_n^N - \mathbb{Q}_n](F_n) \Rightarrow \mathcal{N}(0,\sigma^2_n(F_n)).
$$
\end{theorem}

\begin{proof}
By translation, one can assume that $\QQ_n(F_n)=0$.
For notational convenience, we introduce the rescaled quantity
$\widehat{D}_{p,n}(F_n) = D_{p,n}(F_n)/\eta_p Q_{p,n}(1)$ and its empirical analogue 
$\widehat{D}_{p,n}^N(F_n) = D_{p,n}^N(F_n)/\eta_p Q_{p,n}(1)$ 
for 
$D_{p,n}^N(F_n) = \int \mathcal{M}_p^N(x_p,dx_{0:p-1})\mathcal{Q}_{p,n}(x_p,dx_{p+1:n})F_n(x_{0:n})$.
From \cite[Page 965]{dds1} and Definition \cite[Page 962, eq.~(5.3)]{dds1}, it follows that
\begin{equation*}
\sqrt{N} \, [\mathbb{Q}_n^N - \mathbb{Q}_n](F_n)
=
\sqrt{N} \, 
\sum_{p=0}^n
\frac{\overline{\gamma}_p^N(1)}{\overline{\gamma}_n^N(1)} \, [\eta_p^N-\Phi_p(\eta_{p-1}^N)](\widehat{D}_{p,n}^N(F_n))
\end{equation*}
where we have set $\overline{\gamma}_p^N(1) = \gamma_p^N(1)/\gamma_p(1)$. For brevity, we set
$g_p(x_p) = \widehat{D}_{p,n}(F_n)(x_p)$ and 
$g_p^{N}(x_p)  = \widehat{D}_{p,n}^N(F_n)(x_p)$.
Since the quantity $\overline{\gamma}_p^N(1)$ converges to one in probability (see e.g.~Proposition \ref{prop:wlln}), Slutsky's Lemma shows that one can ignore the term $\overline{\gamma}_p^N(1) / \overline{\gamma}_n^N(1)$ for proving the CLT. The proof consists in exploiting the decomposition
\begin{align*}
\sum_{p=0}^n \sqrt{N}[\eta_p^N-\Phi_p(\eta_{p-1}^N)](g^N_p)
=
\sum_{p=0}^n &\sqrt{N}[\eta_p^N-\Phi_p(\eta_{p-1}^N)](g^N_p - g_p)
+\\
&\sum_{p=0}^n \sqrt{N}[\eta_p^N-\Phi_p(\eta_{p-1}^N)](g_p). 
\end{align*}
and prove that the first term on the R.H.S converges to zero in probability while the second term converges in laws towards a centred Gaussian distribution with variance $\sigma^2_n(F_n)$.
\begin{itemize}
\item
Note that the boundedness assumptions on the potentials $\{G_p\}_{p=1}^n$ and test function $F_n$ imply that $g_p\in \mathbb{B}_b(\mathsf{X})$ for $0\le p\le n$; by  standard results \cite[Corollary 9.3.1]{delmoral}, the sequence 
$
\sqrt{N}\,\big(\,[\eta_0^N-\eta_0](g_0)\, ,\dots, \,
[\eta_n^N-\Phi_n(\eta_{n-1}^N)](g_n)\,\big)
$
converges in laws towards a centred Gaussian vector with covariance matrix 
$\textrm{diag}\big( \, \var_{\eta_0}(g_0),\ldots,\var_{\eta_n}(g_n) \,\big)$. It follows that 
$\sum_{p=0}^n \sqrt{N}[\eta_p^N-\Phi_p(\eta_{p-1}^N)](g_p)$ converges in laws towards a centred Gaussian distribution with variance $\sum_{j=0}^n \var_{\eta_j}(g_j)$; this is just another way of writing $\sigma_n^2(F_n)$.

\item The last part of the proof consists in showing that  the term $\sum_{p=0}^n \sqrt{N}[\eta_p^N-\Phi_p(\eta_{p-1}^N)](g^N_p - g_p)$ converges to zero in probability; this quantity has zero expectation and standard manipulations show that its moment of order two is upper bounded by $\sum_{p=0}^n \mathbb{E}\,\big[\,\Phi_p(\eta_{p-1}^N)(|g_p^N - g_p|^2)\,\big]$. It thus remains to verify that for any index $0 \leq p \leq n$ the quantity 
$\mathbb{E}\,\big[\,\Phi_p(\eta_{p-1}^N)(|g_p^N - g_p|^2)\,\big]$ converges to zero as $N \to \infty$. We use the decomposition
$\Phi_p(\eta_{p-1}^N)(|g_p^N - g_p|^2)
=
\eta_{p}(|g_p^N - g_p|^2)
+
\Phi_p(\eta_{p-1}^N-\eta_{p-1})(|g_p^N - g_p|^2)
$
and treat each term separately. By boundedness of the potntials $\{G_p\}_{p=0}^n$, the quantity $g_p$ and $g_p^N$ are uniformly bounded; it follows from the dominated convergence theorem, Fubini's theorem and Lemma \ref{lem:clt_m} that $\EE[\eta_{p}(|g_p^N - g_p|^2)]$ converges to zero. For dealing with the second term, note that $\Phi_p(\eta_{p-1}^N-\eta_{p-1})(|g_p^N - g_p|^2)$ is less than
\begin{equation}
\int_{\mathsf{X}}\Big|\frac{\eta_{p-1}^N(G_{p-1}H_p(\cdot,x_p))}
{\eta_{p-1}^N(G_{p-1})}-\frac{\eta_{p-1}(G_{p-1}H_p(\cdot,x_p))}{\eta_{p-1}(G_{p-1})}\Big| \times  [g_p^N(x_p) - g_p(x_p) ]^2 \, \lambda(dx_p). \label{eq:end}
\end{equation}
By uniform boundedness of $g_p$ and $g^N_p$ and Fubini's theorem, the conclusion follows once it is established that
\begin{equation}
\int_{\mathsf{X}} \EE\Big|\frac{\eta_{p-1}^N(G_{p-1}H_p(\cdot,x_p))}
{\eta_{p-1}^N(G_{p-1})}-\frac{\eta_{p-1}(G_{p-1}H_p(\cdot,x_p))}{\eta_{p-1}(G_{p-1})}\Big| \, \lambda(dx_p)
\end{equation}
converges to zero. By Assumption \ref{hyp:6} and the boundedness of $G_{p-1}$, for every fixed $x_p \in \mathsf{X}$  Proposition \ref{prop:wlln} applies to the function $G_{p-1}H_p(\cdot,x_p)$ and $G_{p-1}$; it follows that for every fixed $x_p \in \mathsf{X}$  the function
\begin{equation}\label{eq.intermediary.func}
\Big|\frac{\eta_{p-1}^N(G_{p-1}H_p(\cdot,x_p))}
{\eta_{p-1}^N(G_{p-1})}-\frac{\eta_{p-1}(G_{p-1}H_p(\cdot,x_p))}{\eta_{p-1}(G_{p-1})}\Big|\frac{1}{\eta_{p-1}(G_{p-1}H_p(\cdot,x_p))}
\end{equation}
converges to zero in probability. Lemma \ref{lem:ui_term} shows that for $\lambda$-a.e. fixed $x_p\in\mathsf{X}$ the function \eqref{eq.intermediary.func} is also uniformly integrable; consequently, for $\lambda$-a.e. fixed $x_p\in\mathsf{X}$ the function \eqref{eq.intermediary.func} converges in expectation to zero.
In addition, by Lemma \ref{lem:ui_term}
$$
\int_{\mathsf{X}}\mathbb{E}\,\Big|\,\frac{\eta_{p-1}^N(G_{p-1}H_p(\cdot,x_p))}
{\eta_{p-1}^N(G_{p-1})}-\frac{\eta_{p-1}(G_{p-1}H_p(\cdot,x_p))}{\eta_{p-1}(G_{p-1})}\,\Big|\,\lambda(dx_p) \leq
$$
$$
 c \int_{\mathsf{X}}v(x_p)^{2\alpha} \eta_{p-1}(G_{p-1}H_p(\cdot,x_p))\lambda(dx_p).
$$
Application of Fubini  and repeated use of \cite[Lemma 3]{whiteley} allows us to show $\int_{\mathsf{X}}v(x_p)^{2\alpha} \eta_{p-1}(G_{p-1}$ $H_p(\cdot,x_p))\lambda(dx_p)\leq c$, where
 $c<+\infty$ depends on $p$ but not $N$.
Thus, by the dominated convergence theorem, we have shown that the term in (\ref{eq:end})
goes to zero, from which we can conclude the proof.
\end{itemize}
\end{proof}

\begin{rem}
If one wants to adapt the proof for $n$ growing (as in \cite{berard}) the proof as used here must be modified as many of the moment bounds will grow with $n$ (e.g.~Lemma \ref{lem:v_cont}); this is a known problem in SMC, see for instance \cite[Page 20]{beskos1}. This is because we do not control expectations (w.r.t.~the simulated algorithm) of unbounded functions,
uniformly in time. This particular problem is very challenging (for example the work of \cite{douc2,whiteley} do not deal directly with the particle system) and is yet to be handeled in the literature; we do not address this problem. We note also that the proofs of \cite{dds1,douc} also suffer from this deficiency
and assume much stronger hypothesis than in this work.
\end{rem}

\section{Control of the Asymptotic Variance}\label{sec:av_cont}

We now consider the asymptotic variance when $F_n(x_{0:n})=\sum_{p=0}^nf_p(x_p)$, $f_p:\mathsf{X}\rightarrow\mathbb{R}$. Contrary to Theorem \ref{theo:clt} will not assume that the $f_p$ are bounded;
let 
\begin{equation}
\|f\|_{v^{\alpha}}=\sup_{p \geq 0}\|f_p\|_{v^{\alpha}}
\label{eq:f_sup_v}.
\end{equation}

\begin{rem}
In some cases $F_n(x_{0:n})=\sum_{p=0}^nf_p(x_{p-1:p})$ ($x_{-1}$ is null) is of interest. This can be dealt with by either introducing a dirac mass in the Markov kernel $M_n$ and using multistep drift and minorization condtions (see \cite{whiteley} for a discussion),
or with some modifications of the following arguments.
\end{rem}

\begin{theorem}\label{theo:asymp_var_bound}
Assume (A\ref{hyp:1}-\ref{hyp:5}). Then if $\|f\|_{v^{\alpha}}<+\infty$, $\alpha\in(0,1/6)$ 
there exist a $c<+\infty$ which only depends upon the constants in 
(A\ref{hyp:1}), (A\ref{hyp:6}-\ref{hyp:5}) such that for any $n\geq 1$:
$$
\sigma^2(F_n) \leq c\|f\|_{v^{\alpha}}(n+1).
$$
\end{theorem}
\begin{proof}
Recall
$$
\sigma^2_n(F_n) = \sum_{p=0}^n \eta_p\bigg(\bigg[h_{p,n}\Big\{P_{p,n}(F_n)-\frac{\eta_p(D_{p,n}(F_n))}{\eta_p(D_{p,n}(1))}\Big\}\bigg]^2\bigg).
$$
Let us consider the term
$$
h_{p,n}(x)\Big\{P_{p,n}(F_n)(x)-\frac{\eta_p(D_{p,n}(F_n))}{\eta_p(D_{p,n}(1))}\Big\}
$$
in the asymptotic variance expression. 
We have the simple calculation:
$$
P_{p,n}(F_n)(x)-\frac{\eta_p(D_{p,n}(F_n))}{\eta_p(D_{p,n}(1))}
= \frac{(\delta_{x}\otimes\eta_p-\eta_p\otimes\delta_{x})(\overline{D}_{p,n}(F_n\otimes 1))}{\eta_p(D_{p,n}(1))D_{p,n}(1)(x)}
$$
where $\overline{D}_{p,n}=D_{p,n}\otimes D_{p,n}$ and the $\overline{\bullet}$ notation is used to denote operators/functions on the product space.
Then, using the additive nature of the functional $F_n$, one derives:
\begin{eqnarray*}
(\delta_{x_p}\otimes\eta_p-\eta_p\otimes\delta_{x_p})\overline{D}_{p,n}(F_n\otimes 1) & = & \sum_{q=0}^{p-1}(\delta_{x_p}\otimes\eta_p-\eta_p\otimes\delta_{x_p})(\overline{Q}_{p,n}(1) 
\overline{M}_{p:q}(f_q\otimes 1)) \\& &+ \sum_{q=p}^{n}(\delta_{x_p}\otimes\eta_p-\eta_p\otimes\delta_{x_p})(\overline{Q}_{p,q}( (f_q\otimes 1) \overline{Q}_{q,n}(1)))
\end{eqnarray*}
where $M_{p:q} = M_{p,\eta_{p-1}}\dots M_{q+1,\eta_{q}}$. 

We consider first for $p\geq 1$:
$$
\frac{h_{p,n}(x)}{\eta_p(D_{p,n}(1))D_{p,n}(1)(x)} \sum_{q=0}^{p-1}(\delta_{x}\otimes\eta_p-\eta_p\otimes\delta_{x})(\overline{Q}_{p,n}(1) 
\overline{M}_{p:q}(f_q\otimes 1)) = 
$$
$$
\frac{h_{p,n}(x)}{\eta_p(Q_{p,n}(1))} \sum_{q=0}^{p-1}\eta_p(Q_{p,n}(1)[M_{p:q}(f_q)(x) - M_{p:q}(f_q)]).
$$
By Proposition \ref{prop:backward} the R.H.S.~is upper-bounded by
$
c \|f\|_{v^{\alpha}} v(x)^{2\alpha}.
$
Then we consider (which covers the case $p=0$)
$$
\frac{h_{p,n}(x)}{\eta_p(D_{p,n}(1))D_{p,n}(1)(x)}
\sum_{q=p}^{n}(\delta_{x_p}\otimes\eta_p-\eta_p\otimes\delta_{x_p})(\overline{Q}_{p,q}( (f_q\otimes 1) \overline{Q}_{q,n}(1))) = 
$$
$$
\frac{h_{p,n}(x)}{\eta_p(Q_{p,n}(1))Q_{p,n}(1)(x)} \sum_{q=p}^{n}(\delta_{x}\otimes\eta_p-\eta_p\otimes\delta_{x})(\overline{Q}_{p,q}( (f_q\otimes 1) \overline{Q}_{q,n}(1))).
$$
By Proposition \ref{prop:forward}, the R.H.S.~is upper-bounded by
$
c\|f\|_{v^{\alpha}} v(x)^{3\alpha}.
$
Thus, we have proved that
$$
\sigma^2_n(F_n) \leq c\|f\|_{v^{\alpha}} \sum_{p=0}^n \eta_p(v^{6\alpha}).
$$
We conclude by noting $\alpha\in(0,1/6)$ and using \cite[Proposition 1]{whiteley}.
\end{proof}

\section{An Example}\label{sec:verify}

An example where our assumptions can hold, is that of \cite[Section 3.2]{whiteley}, with some minor modifications. We recount the details here.
$\mathsf{X}=\mathbb{R}^{d_x}$ with $n\geq 0$
$$
X_{n+1} = X_{n} + W_n\quad W_n\stackrel{i.i.d.}{\sim} \mathcal{N}_{d_x}(0,I_{d_x})
$$
$I_{d_x}$ the $d_x\times d_x$ identity matrix.
One can take $V(x) = 1 + \frac{x^T x}{2(1+\delta_0)}$, $\delta_0>1$. The observation model is taken as
$$
Y_n |X_n=x \sim \mathcal{N}_{d_y}(H(x),\sigma^2 I_{d_y})
$$
where $H:\mathsf{X}\rightarrow\mathbb{R}^{d_y}$; that is $G_n(x)$ is the $d_y$-dimensional Gaussian density with mean $H(x)$ covariance $I_{d_y}$
and is evaluated point-wise at the observed $y_n$.
It is assumed that the actual observations lie on a space $\mathsf{Y}_{\star}\subset\mathbb{R}^{d_y}$, with $\mathsf{Y}_{\star}$ compact.
If $H$ is bounded such that
$$
\lim_{r\rightarrow\infty}\sup_{|x|\geq r} \frac{x^T x}{2} \frac{1+\delta_1}{\delta_0(1+\delta_0)} + \frac{1}{\sigma^2_y} \sup_{y\in\mathsf{Y}_{\star}}|y|\sup_{|\lambda|=1}\lambda^T H(x) - \frac{H(x)^TH(x)}{2\sigma^2_y}<0
$$
with $\delta_1\in(0,1)$ then one can verify all of the assumptions, including $1/G_{n-1}\in\mathscr{L}_{v^{\delta/2}}$ using the work in \cite{whiteley}, apart from (A\ref{hyp:6}).
This latter assumption will hold, if one can show that for each $\alpha\in(0,1/2)$
\begin{equation}
\inf_{y\in\mathsf{X}}\Big(\big(\inf_{x\in C_d}H_n(x,y)\big)v(y)^{\alpha}\Big)>0  \label{eq:ver_eq}.
\end{equation}
This is because $\eta_{n-1}(C_d)$ can be shown to be lower-bounded uniformly in $n$ (see the proof of \cite[Lemma 8]{whiteley}) and $G_{n-1}$ is (uniform in $n$) upper and lower-bounded if $\mathsf{Y}_{\star}$ is compact (which it is).
Simple calculations show that \eqref{eq:ver_eq} can hold if $\sigma^2_y>4$ and then taking $1<\delta_0$ small enough.

Another observation model (with the above hidden Markov chain and $v(x)$) for which one can verify the assumptions of this article can be found in \cite[Section 3.1.1.]{whiteley}.
Here one sets $\mathsf{Y}_{\star}=\mathsf{Y}=\{0,1\}^{d_x}$ and writing $\mathcal{B}(p)$ as the Bernoulli distribution with success probability $p$, the observation model is
$$
Y_n |X_n=x \sim \mathcal{B}(p(x^1))\otimes\cdots\otimes\mathcal{B}(p(x^{d_x}))
$$
where $p(x) = 1/(1+e^{-x})$. It is easily shown that $1/G_{n-1}\in\mathscr{L}_{v^{\delta/2}}$ and all the other assumptions apart from (A\ref{hyp:6}) easily follow. The latter assumption will follow by the above calculations
and the fact that (treating $G_n$ as a function of the observations also) $G_n(x;y)\leq 1$ and $\inf_{(y,x)\in\mathsf{Y}\times C_d}G_n(x;y)>0$.

\subsubsection*{Acknowledgements}
The author was supported by Singapore MOE grant R-155-000-119-133.

\appendix

\section{Technical Results for Central Limit Theorem}

Throughout this Section we suppose that for any $n\geq 0$, $\|G_n\|_{\infty}<+\infty$ and this is ommited from all statements below. We also use $\mathbb{E}[\cdot]$ to denote expectation w.r.t.~the particle system.
$\mathscr{F}_{n}^N$ is the natural filtration of the particles at time $n$. 

\begin{lem}\label{lem:clt_m}
Assume (A\ref{hyp:1}-\ref{hyp:6}).
Suppose that for each $n\geq 0$, $1/G_{n}\in\mathscr{L}_{v^{\delta}}$, with $\delta$ as in (A\ref{hyp:1}).
 Let $p>0$, then for $\lambda-$a.e.~$x_p\in\mathsf{X}$ and any $F\in\mathbb{B}_b(\mathsf{X}^{n+1})$
$$
[D_{p,n}^N-D_{p,n}](F)(x_p) \rightarrow_{\mathbb{P}} 0.
$$
\end{lem}

\begin{proof}
By \cite[Lemma 6.1]{dds1}, we have
\begin{equation}
[D_{p,n}^N-D_{p,n}](F)(x_p) = \sum_{q=0}^p [\mathcal{M}_{p,q,\eta_q^N}-\mathcal{M}_{p,q,\Phi_q(\eta_{q-1}^N)}](S_{p,q,n}^N(F))(x_p)\label{eq:tech_lem1}
\end{equation}
where for $\mu\in\mathcal{P}$, $0\leq q <p$
$$
\mathcal{M}_{p,q,\mu}(x_p,dx_{q:p-1}) = \frac{\mu(dx_q)\mathcal{Q}_{q,p-1}(x_q,dx_{q+1:p-1}) G_{p-1}(x_{p-1})H_p(x_{p-1},x_p)}{\mu Q_{q,p-1}(G_{p-1}H_p(\cdot,x_p))}
$$
$\mathcal{Q}_{q,p-1}$ is defined in \eqref{eq:q_path_def} and
$$
S_{p,q,n}^N(F)(x_{q:p}) = \int_{\mathsf{X}^{q+n-p}} \mathcal{Q}_{p,n}(x_p,dx_{p+1:n}) \mathcal{M}_q^N(x_q,dx_{0:q-1})F(x_{0:n})
$$
see \eqref{eq:m_back_path} for a defintion of $\mathcal{M}_q^N$.
We note that 
\begin{equation}
\sup_{x_{q:p}\in\mathsf{X}^{p-q+1}}|S_{p,q,n}^N(F)(x_{q:p})|\leq c\|F\|_{\infty}
\label{eq:tech_lem2}
\end{equation} 
where $c$ is a finite constant that may depend on $p,n$ but not $N$.
We will show that each summand on the R.H.S.~of \eqref{eq:tech_lem1} will converge to zero in probability.

It is first remarked that by (A\ref{hyp:1}), (A\ref{hyp:6}) and Proposition \ref{prop:wlln} 
$$
\eta_{q}^NQ_{q,p-1}\Big(\frac{G_{p-1}H_p(\cdot,x_p)}{\eta_{p-1}(G_{p-1}H_p(\cdot,x_p))}\Big) \rightarrow_{\mathbb{P}}
\eta_{q}Q_{q,p-1}\Big(\frac{G_{p-1}H_p(\cdot,x_p)}{\eta_{p-1}(G_{p-1}H_p(\cdot,x_p))}\Big)
$$
and
$$
\Phi_q(\eta_{q-1}^N)Q_{q,p-1}\Big(\frac{G_{p-1}H_p(\cdot,x_p)}{\eta_{p-1}(G_{p-1}H_p(\cdot,x_p))}\Big) \rightarrow_{\mathbb{P}}
\eta_{q}Q_{q,p-1}\Big(\frac{G_{p-1}H_p(\cdot,x_p)}{\eta_{p-1}(G_{p-1}H_p(\cdot,x_p))}\Big)
$$
so it is enough to show that 
$$
\eta_{q}Q_{q,p-1}\Big(\frac{G_{p-1}H_p(\cdot,x_p)}{\eta_{p-1}(G_{p-1}H_p(\cdot,x_p))}\Big)^{-1}
\Big(
[\eta_q^N-\Phi_q(\eta_{q-1}^N)]\bigg[\mathcal{Q}_{q,p-1}\Big(\frac{G_{p-1}H_p(\cdot,x_p)}{\eta_{p-1}(G_{p-1}H_p(\cdot,x_p))}S_{p,q,n}^N(F)\Big)\bigg]
\Big)
$$
converges in probability to zero. We have via Jensen and the (condtional) Marcinkiewicz-Zygmund inequalities that
$$
\mathbb{E}\Bigg[\Bigg|
[\eta_q^N-\Phi_q(\eta_{q-1}^N)]\bigg[\mathcal{Q}_{q,p-1}\Big(\frac{G_{p-1}H_p(\cdot,x_p)}{\eta_{p-1}(G_{p-1}H_p(\cdot,x_p))}S_{p,q,n}^N(F)\Big)\bigg]
\Bigg|\Bigg] \leq 
$$
$$
\frac{c}{\sqrt{N}}\mathbb{E}\Bigg[\Bigg|\mathcal{Q}_{q,p-1}\Big(\frac{G_{p-1}H_p(\cdot,x_p)}{\eta_{p-1}(G_{p-1}H_p(\cdot,x_p))}S_{p,q,n}^N(F)\Big)(x_q^1)\Bigg|^2\Bigg]^{1/2}.
$$
By \eqref{eq:tech_lem2}
$$
\mathbb{E}\Bigg[\Bigg|\mathcal{Q}_{q,p-1}\Big(\frac{G_{p-1}H_p(\cdot,x_p)}{\eta_{p-1}(G_{p-1}H_p(\cdot,x_p))}S_{p,q,n}^N(F)\Big)(X_q^1)\Bigg|^2\Bigg]^{1/2} \leq 
c\|F\|_{\infty}\mathbb{E}\bigg[Q_{q,p-1}\Big(\frac{G_{p-1}H_p(\cdot,x_p)}{\eta_{p-1}(G_{p-1}H_p(\cdot,x_p))}\Big)(X_q^1)^2\bigg]^{1/2}.
$$
Then by (A\ref{hyp:6}) and repeated application of \cite[Lemma 3]{whiteley}, we have
$$
\mathbb{E}\bigg[Q_{q,p-1}\Big(\frac{G_{p-1}H_p(\cdot,x_p)}{\eta_{p-1}(G_{p-1}H_p(\cdot,x_p))}\Big)(X_q^1)^2\bigg]^{1/2} \leq  c v(x_p)^{\alpha}\mathbb{E}[v(X_q^1)^{2\alpha}]^{1/2}
$$
then, for $\mathbb{E}[v(X_q^1)^{2\alpha}]$, Jensen and application of Lemma \ref{lem:v_cont}, yields that
$$
\mathbb{E}\Bigg[\Bigg|
[\eta_q^N-\Phi_q(\eta_{q-1}^N)]\bigg[\mathcal{Q}_{q,p-1}\Big(\frac{G_{p-1}H_p(\cdot,x_p)}{\eta_{p-1}(G_{p-1}H_p(\cdot,x_p))}S_{p,q,n}^N(F)\Big)\bigg]
\Bigg|\Bigg] \leq \frac{c}{\sqrt{N}}  v(x_p)^{\alpha}.
$$
Thus we have shown that
$$
\eta_{q}Q_{q,p-1}\Big(\frac{G_{p-1}H_p(\cdot,x_p)}{\eta_{p-1}(G_{p-1}H_p(\cdot,x_p))}\Big)^{-1}
\Big(
[\eta_q^N-\Phi_q(\eta_{q-1}^N)]\bigg[\mathcal{Q}_{q,p-1}\Big(\frac{G_{p-1}H_p(\cdot,x_p)}{\eta_{p-1}(G_{p-1}H_p(\cdot,x_p))}S_{p,q,n}^N(F)\Big)\bigg]
\Big)
$$
converges in probability to zero, from which we can conclude.
\end{proof}

\begin{lem}\label{lem:ui_term}
Assume (A\ref{hyp:1}-\ref{hyp:6}). Suppose that for each $n\geq 0$, $1/G_{n}\in\mathscr{L}_{v^{\delta/2}}$, with $\delta$ as in (A\ref{hyp:1}), then 
there exist a $1\geq \upsilon>0$ such that
for any $n\geq 1$ there exist a $c<+\infty$ such that for $\lambda-$a.e.~$x_n\in\mathsf{X}$
$$
\mathbb{E}\Bigg[\Bigg|\Bigg[\frac{\eta_{n-1}^N(G_{n-1}H_n(\cdot,x_n))}
{\eta_{n-1}^N(G_{n-1})}-\frac{\eta_{n-1}(G_{n-1}H_p(\cdot,x_n))}{\eta_{n-1}(G_{n-1})}\Bigg]\frac{1}{\eta_{n-1}(G_{n-1}H_n(\cdot,x_n))}\Bigg|^{1+\upsilon}\Bigg]
\leq c v(x_n)^{(1+\upsilon)\alpha}
$$ 
where $\alpha$ is as in (A\ref{hyp:6}).
\end{lem}

\begin{proof}
Throughout $c$ is a constant whose value can change from line to line, but only depends upon $n$.
We have 
$$
\mathbb{E}\Bigg[\Bigg|\Bigg[\frac{\eta_{n-1}^N(G_{n-1}H_n(\cdot,x_n))}
{\eta_{n-1}^N(G_{n-1})}-\frac{\eta_{n-1}(G_{n-1}H_p(\cdot,x_n))}{\eta_{n-1}(G_{n-1})}\Bigg]\frac{1}{\eta_{n-1}(G_{n-1}H_n(\cdot,x_n))}\Bigg|^{1+\upsilon}\Bigg]
\leq
$$
$$
c\Bigg(\frac{1}{\eta_{n-1}(G_{n-1})^{1+\upsilon}}
+ \mathbb{E}\Bigg[\Bigg|\frac{\eta_{n-1}^N(G_{n-1}H_n(\cdot,x_n))}
{\eta_{n-1}^N(G_{n-1})\eta_{n-1}(G_{n-1}H_n(\cdot,x_n))}\Bigg|^{1+\upsilon}\Bigg]
\Bigg).
$$
Then, application of (A\ref{hyp:6}) gives that
\begin{equation}
\mathbb{E}\Bigg[\Bigg|\frac{\eta_{n-1}^N(G_{n-1}H_n(\cdot,x_n))}
{\eta_{n-1}^N(G_{n-1})\eta_{n-1}(G_{n-1}H_n(\cdot,x_n))}\Bigg|^{1+\upsilon}\Bigg] \leq cv(x_n)^{(1+\upsilon)\alpha}
\mathbb{E}\Bigg[\Bigg|\frac{\eta_{n-1}^N(v^{\alpha})}
{\eta_{n-1}^N(G_{n-1})}\Bigg|^{1+\upsilon}\Bigg]\label{eq:lem_ui1}.
\end{equation}

We will show now that (see the R.H.S.~of \eqref{eq:lem_ui1})
$$
\mathbb{E}\Bigg[\Bigg|\frac{\eta_{n-1}^N(v^{\alpha})}
{\eta_{n-1}^N(G_{n-1})}\Bigg|^{1+\upsilon}\Bigg] \leq c
$$
for some $1\geq\upsilon>0$
when $\alpha=1/2$ (recall $\alpha\in(0,1/2)$). From the proof of Lemma \ref{lem:v_cont}, equation \eqref{eq:ui_eq_needed} one can show in a similar manner that
$$
\mathbb{E}\Bigg[\Bigg|\frac{\eta_{n-1}^N(v^{\frac{1}{2}})}
{\eta_{n-1}^N(G_{n-1})}\Bigg|^{1+\upsilon}\Bigg]  \leq c\mathbb{E}\big[\big\{\eta_{n-1}^N(v^{\frac{1}{2}})\eta_{n-1}^N(v^{\delta/2})\big\}^{1+\upsilon}\big].
$$
Then we have by Minkowski
$$
\mathbb{E}\big[\big\{\eta_{n-1}^N(v^{\frac{1}{2}})\eta_{n-1}^N(v^{\delta/2})\big\}^{1+\upsilon}\big] \leq
$$
$$
\frac{1}{N^{2(1+\upsilon)}}\Big(\mathbb{E}\big[\big\{\sum_i v(X_{n-1}^i)^{\frac{1+\delta}{2}}\big\}^{1+\upsilon}\big]^{\frac{1}{1+\upsilon}} + 
\mathbb{E}\big[\big\{\sum_{i\neq j} v(X_{n-1}^i)^{\frac{1}{2}}v(X_{n-1}^j)^{\frac{\delta}{2}}\big\}^{1+\upsilon}\big]^{\frac{1}{1+\upsilon}}
\Big)^{1+\upsilon} \leq
$$
$$
\frac{1}{N^{2(1+\upsilon)}}\Big(N \mathbb{E}\big[v(X_{n-1}^1)^{\frac{(1+\delta)(1+\upsilon)}{2}}\big]^{\frac{1}{1+\upsilon}} +
\frac{N(N-1)}{2}\mathbb{E}\big[v(X_{n-1}^1)^{\frac{1}{2}}v(X_{n-1}^2)^{\frac{\delta(1+\upsilon)}{2}}\big]^{\frac{1}{1+\upsilon}}
\Big)^{1+\upsilon}.
$$
Let $0<\upsilon<(1-\delta)/(1+\delta)$, we will show that two expectations in the line above are upper-bounded by a constant. For 
$$
\mathbb{E}\big[v(X_{n-1}^1)^{\frac{(1+\delta)(1+\upsilon)}{2}}\big]^{\frac{1}{1+\upsilon}}
$$
one can apply Jensen followed by Lemma \ref{lem:v_cont}. For
$$
\mathbb{E}\big[v(X_{n-1}^1)^{\frac{1}{2}}v(X_{n-1}^2)^{\frac{\delta(1+\upsilon)}{2}}\big]^{\frac{1}{1+\upsilon}}
$$
we can apply Cauchy-Schwarz to obtain the upper-bound
$$
\mathbb{E}[v(X_{n-1}^1)]^{\frac{1}{2(1+\upsilon)}}\mathbb{E}[v(X_{n-1}^2)^{\delta(1+\upsilon)}]^{\frac{1}{2(1+\upsilon)}}
$$
the left hand expectation is controlled via  Lemma \ref{lem:v_cont} and the right-hand via Jensen followed by Lemma \ref{lem:v_cont}. Hence one can deduce that
$$
\mathbb{E}\Bigg[\Bigg|\frac{\eta_{n-1}^N(v^{\frac{1}{2}})}
{\eta_{n-1}^N(G_{n-1})}\Bigg|^{1+\upsilon}\Bigg] \leq c
$$
for some $\upsilon>0$ which concludes the proof of the Lemma.
\end{proof}

\begin{prop}\label{prop:wlln}
Assume (A1-2). Suppose that for each $n\geq 0$, $1/G_{n}\in\mathscr{L}_{v^{\delta}}$, with $\delta$ as in (A\ref{hyp:1}), then for any 
$\varrho>0$, $f\in\mathscr{L}_{v^{1/(1+\varrho)}}$, $n\geq 0$
$$
\eta_n^N(f) \rightarrow_{\mathbb{P}} \eta_n(f).
$$
\end{prop}

\begin{proof}
The result is proved by induction.
The case $n=0$ follows by the weak law of large numbers for i.i.d.~random variables; $\eta_0\in\mathcal{P}_v$. Thus, the result is assumed for $n-1$ and we consider $n$.
We have
\begin{equation}
[\eta^N_{n}-\eta_n](f) = [\eta^N_{n}-\Phi_n(\eta^N_{n-1})](f) + [\Phi_n(\eta^N_{n-1})-\eta_n](f)\label{eq:wlln1}.
\end{equation}

We first deal with the second term on the R.H.S.~of \eqref{eq:wlln1}. We have the standard decomposition
$$
[\Phi_n(\eta^N_{n-1})-\eta_n](f) = \Big[\frac{1}{\eta_{n-1}^N(G_{n-1})}-\frac{1}{\eta_{n-1}(G_{n-1})}\Big]\eta_{n-1}^N(Q_n(f)) + 
\frac{1}{\eta_{n-1}(G_{n-1})}[\eta^N_{n}-\eta_n](Q_n(f)).
$$
By the proof of \cite[Lemma 3]{whiteley} $Q_n(f)\in\mathscr{L}_{v^{1/(1+\varrho)}}$ (recall that for any $n\geq 0$, $\|G_n\|_{\infty}<+\infty$), so by the induction hypothesis, it follows that 
\begin{equation}
[\Phi_n(\eta^N_{n-1})-\eta_n](f)\rightarrow_{\mathbb{P}}0\label{eq:wlln2}.
\end{equation}

We now deal with the first term on the R.H.S.~of \eqref{eq:wlln1}. One can use \cite[Theorem A.1]{douc1}, which can be applied by Lemma \ref{lem:v_cont}.
We have to verify Eq.~25 and Eq.~26 of that paper: in the notation of this article, they read:
\begin{itemize}
\item{$\sup_{N}\mathbb{P}(\Phi_{n,N}(\eta_{n-1}^N)(|f|)\geq \kappa)\rightarrow 0$ as $\kappa\rightarrow\infty$.}
\item{$\frac{1}{N}\sum_{i=1}^N\mathbb{E}\,\big[\,|f(x_n^i)|\,\mathbb{I}_{\{ |f(x_n^i)|/N\geq \epsilon \}}\big|\,\mathscr{F}_{n-1}^N\,\big] \rightarrow_{\mathbb{P}} 0$, for any $\epsilon>0$.}
\end{itemize}
The tightness condition (i.e.~the first bullet point), Eq.~25, readily follows from equation \eqref{eq:wlln2}. For the second bullet point, set $0<\upsilon\leq \varrho \wedge \delta/(1-\delta)$, one easily has 
$$
\frac{1}{N}\sum_{i=1}^N\mathbb{E}\,\big[\,|f(x_n^i)|\,\mathbb{I}_{\{ |f(x_n^i)|/N\geq \epsilon \}}\big|\,\mathscr{F}_{n-1}^N\,\big]
\leq \Phi_n(\eta_{n-1}^N)(|f|^{1+\upsilon})\frac{1}{(\epsilon N)^{\upsilon}}.
$$
As $Q_n(|f|^{1+\upsilon})\in\mathscr{L}_{v^{1/(1+\varrho)}}$ by construction, it follows that 
$$
\Phi_n(\eta_{n-1}^N)(|f|^{1+\upsilon})\frac{1}{(\epsilon N)^{\upsilon}} \rightarrow_{\mathbb{P}} 0
$$
which completes the proof.

\end{proof}

\begin{lem}\label{lem:v_cont}
Assume (A1-2). Suppose that for each $n\geq 0$, $1/G_{n}\in\mathscr{L}_{v^{\delta}}$, with $\delta$ as in (A\ref{hyp:1}), then for any $n\geq 0$ there exists a $c<+\infty$ such that for any $N\geq 2$
\begin{eqnarray}
\mathbb{E}[v(X_n^1)] & \leq & c \label{eq:v1}
\end{eqnarray}
\end{lem}

\begin{proof}
We proceed via induction.  The case $n=0$ follows as $\eta_0\in\mathcal{P}_v$. Thus, we assume for $n-1$ and consider $n$:
$$
\mathbb{E}[v(X_n^1)] = \mathbb{E}\Big[\frac{\eta_{n-1}^N(Q_n(v))}{\eta_{n-1}^N(G_{n-1})}\Big].
$$
Now, consider
\begin{eqnarray}
\eta_{n-1}^N(G_{n-1}) & = & \eta_{n-1}^N\Big(G_{n-1}v^{\delta}\frac{1}{v^{\delta}}\Big)\nonumber\\
&\geq &  \|1/G_{n-1}\|_{v^{\delta}}^{-1} \eta_{n-1}^N\Big(\frac{1}{v^{\delta}}\Big) \nonumber\\
&\geq & \|1/G_{n-1}\|_{v^{\delta}}^{-1}\frac{1}{\eta_{n-1}^N(v^{\delta})}\label{eq:ui_eq_needed}.
\end{eqnarray}
So, we have that
$$
\mathbb{E}[v(X_n^1)] \leq \|1/G_{n-1}\|_{v^{\delta} }  \mathbb{E}[\eta_{n-1}^N(Q_n(v))\eta_{n-1}^N(v^{\delta})].
$$
Now via the multiplicative drift $Q_n(v) \leq c v^{1-\delta}$, so
\begin{eqnarray*}
\mathbb{E}[\eta_{n-1}^N(Q_n(v))\eta_{n-1}^N(v^{\delta})] & \leq & c\mathbb{E}\Big[\frac{1}{N^2}\Big(\sum_{i} v(X_{n-1}^i) + \sum_{i\neq j} v(X_{n-1}^i)^{1-\delta} v(X_{n-1}^j)^{\delta} \Big)\Big]\\
& = &  c \Big(\mathbb{E}[v(X_{n-1}^1)] + \frac{N-1}{N}\mathbb{E}[v(X_{n-1}^1)^{1-\delta} v(X_{n-1}^2)^{\delta}]\Big) \\
&\leq & 2c \mathbb{E}[v(X_{n-1}^1)]
\end{eqnarray*}
where we have applied H\"older to get to the last line; the induction hypothesis completes the proof of \eqref{eq:v1}.

\end{proof}

\section{Proofs for the Asymptotic Variance}

We give the proofs which are used for Theorem \ref{theo:asymp_var_bound}, bounding the asymptotic variance. This is broken into three sections: controlling the forward part of the asymptotic variance:
$$
\frac{h_{p,n}(x)}{\eta_p(Q_{p,n}(1))Q_{p,n}(1)(x)} \sum_{q=p}^{n}(\delta_{x}\otimes\eta_p-\eta_p\otimes\delta_{x})(\overline{Q}_{p,q}( (f_q\otimes 1) \overline{Q}_{q,n}(1)))
$$
controlling the backward part of the asymptotic variance
$$
\frac{h_{p,n}(x)}{\eta_p(Q_{p,n}(1))} \sum_{q=0}^{p-1}\eta_p(Q_{p,n}(1)[M_{p:q}(f_q)(x) - M_{p:q}(f_q)])
$$
and the technical results used to achieve this. Recall $\|f\|_{v^{\alpha}}$ is defined in \eqref{eq:f_sup_v}.

The following additional notations are used in this Appendix.
We write $\overline{\mathbb{E}}_{\mu\otimes\eta}$ as the expectation w.r.t.~the inhomogeneous Markov chain
$\{\overline{X}_p\}_{p\geq 0}$ on $\overline{\mathsf{X}}:=\mathsf{X}^2$ with initial distribution $\mu\otimes \eta$ and transition $H_p(x_{p-1},x_p)H_p(y_{p-1},y_p)\lambda(dx_p)\otimes\lambda(dy_p)$. We also use the notation $\overline{M}_{p,q}^{d}:=\sum_{k=p}^{q-1}\mathbb{I}_{\overline{C}_d}(\overline{X}_k)\mathbb{I}_{\overline{C}_d}(\overline{X}_{k+1})$.

\subsection{Controlling the Forward Part}

\begin{prop}\label{prop:forward}
Assume (A\ref{hyp:1}-\ref{hyp:3}), (A\ref{hyp:2}-\ref{hyp:5}). Then if $\|f\|_{v^{\alpha}}<+\infty$, $\alpha\in(0,1/3)$ there exist a $c<+\infty$ 
and $\rho\in(0,1)$ which depends only upon the constants in (A\ref{hyp:1}), (A\ref{hyp:2}-\ref{hyp:5}), such that for any $x\in\mathsf{X}$
\begin{equation}
\frac{h_{p,n}(x)}{\eta_p(Q_{p,n}(1))Q_{p,n}(1)(x)} \sum_{q=p}^{n}(\delta_{x}\otimes\eta_p-\eta_p\otimes\delta_{x})(\overline{Q}_{p,q}( (f_q\otimes 1) \overline{Q}_{q,n}(1)))
\leq c \|f\|_{v^{\alpha}} v(x)^{3\alpha}\Big\{1+ \frac{\rho(1-\rho^{n-p})}{1-\rho}\Big\}.
\label{eq:forward_master_ineq}
\end{equation}
\end{prop}

\begin{proof}
We break up our proof into controlling the summands on the L.H.S.~of \eqref{eq:forward_master_ineq}.

\textbf{Case $q=p$}. We first consider the case $q=p$ in the summation on the L.H.S.~of \eqref{eq:forward_master_ineq}. Then we have 
$$
(\delta_{x}\otimes\eta_p-\eta_p\otimes\delta_{x})(\overline{Q}_{p,q}( (f_q\otimes 1) \overline{Q}_{q,n}(1))) =
(\delta_{x}\otimes\eta_p-\eta_p\otimes\delta_{x})( (f_p\otimes 1) \overline{Q}_{p,n}(1)).
$$
Then as $f_p\in\mathscr{L}_{v^{\alpha}}$, we have
\begin{equation}
\delta_{x}\otimes\eta_p( (f_q\otimes 1) \overline{Q}_{p,n}(1)) \leq  \|f\|_{v^{\alpha}} v(x)^{\alpha} Q_{p,n}(1)(x) \eta_p(Q_{p,n}(1)).
\label{eq:forward_eq1}
\end{equation}
Thus by using a similar argument to \eqref{eq:forward_eq1}
$$
(\delta_{x}\otimes\eta_p-\eta_p\otimes\delta_{x})( (f_p\otimes 1) \overline{Q}_{p,n}(1)) \leq c \|f\|_{v^{\alpha}} Q_{p,n}(1)(x) [v(x)^{\alpha} \eta_p(Q_{p,n}(1))
+ \eta_p(v^{\alpha} Q_{p,n}(1))].
$$
Hence, we have that 
$$
\frac{h_{p,n}(x)}{\eta_p(Q_{p,n}(1))Q_{p,n}(1)(x)} (\delta_{x}\otimes\eta_p-\eta_p\otimes\delta_{x})( (f_p\otimes 1) \overline{Q}_{p,n}(1)) \leq
$$
\begin{equation}
c\|f\|_{v^{\alpha}}\frac{h_{p,n}(x)}{\eta_p(Q_{p,n}(1))} [v(x)^{\alpha} \eta_p(Q_{p,n}(1))
+ \eta_p(v^{\alpha} Q_{p,n}(1))]\label{eq:forward_q=p}
\end{equation}
Now for the first term on the R.H.S.~of \eqref{eq:forward_q=p} we have
$$
\frac{h_{p,n}(x)}{\eta_p(Q_{p,n}(1))} v(x)^{\alpha} \eta_p(Q_{p,n}(1)) \leq c v(x)^{2\alpha}
$$
where we have used Propositions 1 and 2 and Lemma 3 of \cite{whiteley}, i.e.~that $\sup_{n\geq 1}\sup_{0\leq p \leq n}\|h_{p.n}\|_{v^{\alpha}}<+\infty$.
For the second term on the R.H.S.~of \eqref{eq:forward_q=p} we have for any $r\in[\underline{d},\infty)$
$$
\frac{h_{p,n}(x)}{\eta_p(Q_{p,n}(1))}\eta_p(v^{\alpha} Q_{p,n}(1)) = h_{p,n}(x) \eta_p(v^{\alpha} h_{p,n}) \leq 
c v(x)^{\alpha} \eta_p(v^{2\alpha})
$$
where we again use $\sup_{n\geq 1}\sup_{0\leq p \leq n}\|h_{p.n}\|_{v^{\alpha}}<+\infty$.
By Proposition 1 of \cite{whiteley} $\sup_{p\geq 0}\|\eta_p(v^{2\alpha})\|_{v^{\alpha}} < +\infty$, thus
$$
\frac{h_{p,n}(x)}{\eta_p(Q_{p,n}(1))}\eta_p(v^{\alpha} Q_{p,n}(1)) \leq c v(x)^{\alpha}.
$$
Thus for the case $q=p$ we have established that
\begin{equation}
\frac{h_{p,n}(x)}{\eta_p(Q_{p,n}(1))Q_{p,n}(1)(x)} (\delta_{x}\otimes\eta_p-\eta_p\otimes\delta_{x})( (f_p\otimes 1) \overline{Q}_{p,n}(1))
\leq c\|f\|_{v^{\alpha}} v(x)^{2\alpha}.
\label{eq:forward_q=p_main_eq}
\end{equation}

\textbf{Case $q=n$}. Second, we consider the case $q=n$ in the summation on the L.H.S.~of \eqref{eq:forward_master_ineq}. Then we have 
$$
(\delta_{x}\otimes\eta_p-\eta_p\otimes\delta_{x})(\overline{Q}_{p,q}( (f_q\otimes 1) \overline{Q}_{q,n}(1))) =
(\delta_{x}\otimes\eta_p-\eta_p\otimes\delta_{x})( \overline{Q}_{p,n}((f_p\otimes 1)) ).
$$
Then, one can apply the proof of Theorem 1 of \cite{whiteley} to show that there exist a $\rho\in(0,1)$ (which depends upon the the constants in (A\ref{hyp:1}-\ref{hyp:2}), (A\ref{hyp:4}-\ref{hyp:5}))
$$
\frac{(\delta_{x}\otimes\eta_p-\eta_p\otimes\delta_{x})( \overline{Q}_{p,n}((f_p\otimes 1)) )}{\eta_p(Q_{p,n}(1))Q_{p,n}(1)(x)}
\leq c\|f\|_{v^{\alpha}} \frac{v_{p,n,\alpha}(x)}{\|h_{p,n}\|_{v^{\alpha}}}\mu(v^{\alpha}) \rho^{n-p}
$$
where  $v_{p,n,\alpha}(x) = v(x)^{\alpha}\|h_{p,n}\|_{v^{\alpha}}/h_{p,n}(x)$. Thus we have established for $q=n$:
\begin{equation}
h_{p,n}(x)
\frac{(\delta_{x}\otimes\eta_p-\eta_p\otimes\delta_{x})( \overline{Q}_{p,n}((f_p\otimes 1)) )}{\eta_p(Q_{p,n}(1))Q_{p,n}(1)(x)}
\leq c\|f\|_{v^{\alpha}} v(x)^{\alpha}\mu(v^{\alpha}) \rho^{n-p}.
\label{eq:forward_q=n_main_eq}
\end{equation}

\textbf{Case $p<q<n$}. Lastly, we consider the case $p<q<n$ in the summation on the L.H.S.~of \eqref{eq:forward_master_ineq}. Using almost the same calculations
as \cite{whiteley} Theorem 1 (which themselves rely on the proofs of \cite{douc,klep}) we have for arbitrary $d$, $\beta\in(0,1)$:
$$
(\delta_{x}\otimes\eta_p-\eta_p\otimes\delta_{x})(\overline{Q}_{p,q}( (f_q\otimes 1) \overline{Q}_{q,n}(1)))
\leq  2 \|f\|_{v^{\alpha}}\bigg\{ \overline{\mathbb{E}}_{\delta_x\otimes\eta_p}[\prod_{s=p}^{q-1} \overline{G}_q(\overline{X}_s)
\overline{v}(\overline{X}_q)^{\alpha} \overline{Q}_{q,n}(1)(\overline{X}_{q})\times
 $$
\begin{equation}
 \mathbb{I}_{\{\overline{M}_{p,q}^{d}\geq \beta(q-p)\}}\rho_d^{\overline{M}_{p,q}^{d}}] +
\overline{\mathbb{E}}_{\delta_x\otimes\eta_p}[\prod_{s=p}^{q-1} \overline{G}_q(\overline{X}_s)
\overline{v}(\overline{X}_q)^{\alpha} \overline{Q}_{q,n}(1)(\overline{X}_{q})
 \mathbb{I}_{\{\overline{M}_{p,q}^{d}<\beta(q-p)\}}\rho_d^{\overline{M}_{p,q}^{d}}]
\bigg\}
\label{eq:forward_p<q<n}
\end{equation}
where $\rho_d = 1 - \bigg(\frac{\epsilon_d^-}{\epsilon_d^+}\bigg)^2$.
We begin by considering the first term on the R.H.S.~of \eqref{eq:forward_p<q<n}, when multiplied by the term outside the summation on the L.H.S.~of \eqref{eq:forward_master_ineq}.
As in Theorem 1 of \cite{whiteley} as $\rho_d<1$ we have:
$$
\frac{h_{p,n}(x)}{\eta_p(Q_{p,n}(1))Q_{p,n}(1)(x)} \overline{\mathbb{E}}_{\delta_x\otimes\eta_p}[\prod_{s=p}^{q-1} \overline{G}_q(\overline{X}_s)
\overline{v}(\overline{X}_q)^{\alpha} \overline{Q}_{q,n}(1)(\overline{X}_{q})
 \mathbb{I}_{\{\overline{M}_{p,q}^{d}\geq \beta(q-p)\}}\rho_d^{\overline{M}_{p,q}^{d}}]
\leq
$$
$$
 \rho_d^{\beta(q-p)} h_{p,n}(x) \frac{Q_{p,q}(v^{\alpha} Q_{q,n}(1))(x)}{Q_{p,q}(Q_{q,n}(1))(x)} \frac{\eta_p[Q_{p,q}(v^{\alpha} Q_{q,n}(1))]}{\eta_p[Q_{p,q}(Q_{q,n}(1))]}.
$$
Then, one can apply Lemma \ref{lem:easy_part_forward}, to show that
$$
\frac{h_{p,n}(x)}{\eta_p(Q_{p,n}(1))Q_{p,n}(1)(x)} \overline{\mathbb{E}}_{\delta_x\otimes\eta_p}[\prod_{s=p}^{q-1} \overline{G}_q(\overline{X}_s)
\overline{v}(\overline{X}_q)^{\alpha} \overline{Q}_{q,n}(1)(\overline{X}_{q})
 \mathbb{I}_{\overline{M}_{p,q}^{d}\geq \beta(q-p)}\rho_d^{\overline{M}_{p,q}^{d}}]
\leq c \|f\|_{v^{\alpha}}\rho_d^{\beta(q-p)} v(x)^{3\alpha}
$$
Now consider the second term on the R.H.S.~of \eqref{eq:forward_p<q<n}, when multiplied by the term outside the summation on the L.H.S.~of \eqref{eq:forward_master_ineq}. We have
$$
\frac{h_{p,n}(x)\overline{\mathbb{E}}_{\delta_x\otimes\eta_p}\Big[\Big\{\prod_{s=p}^{q-1} \overline{G}_s(\overline{X}_s)\Big\}
\overline{v}(\overline{X}_q)^{\alpha} \overline{Q}_{q,n}(1)(\overline{X}_q)
\mathbb{I}_{ \overline{M}_{p,q}^{d} <\beta(q-p) }
\Big] }{Q_{p,n}(1)(x)\eta_p(Q_{p,n}(1))} \leq 
$$
$$
 c(d,\alpha,\beta) \mu(v^{3\alpha})v(x)^{3\alpha} \exp\{-(q-p)c(d,\alpha,\beta)]\}.
$$
where we note that $d$ was arbitrary above and we have applied Lemma \ref{lem:forward_hard_part}. Then, one can make $d$ larger so that we have
for $p<q<n$:
\begin{equation}
\frac{h_{p,n}(x)}{\eta_p(Q_{p,n}(1))Q_{p,n}(1)(x)}
(\delta_{x}\otimes\eta_p-\eta_p\otimes\delta_{x})(\overline{Q}_{p,q}( (f_q\otimes 1) \overline{Q}_{q,n}(1)))
\leq c \|f\|_{v^{\alpha}}  \rho^{q-p} v(x)^{3\alpha}
\label{eq:forward_p<q<n_main_eq}
\end{equation}
where $\rho\in(0,1)$ depends upon the constants in (A\ref{hyp:1}), (A\ref{hyp:3}-\ref{hyp:5}) as well as $\alpha$.

Then, combining \eqref{eq:forward_q=p_main_eq}, \eqref{eq:forward_q=n_main_eq} and \eqref{eq:forward_p<q<n_main_eq}, we have proved that
for any $x\in\mathsf{X}$
$$
\frac{h_{p,n}(x)}{\eta_p(Q_{p,n}(1))Q_{p,n}(1)(x)} \sum_{q=p}^{n}(\delta_{x}\otimes\eta_p-\eta_p\otimes\delta_{x})(\overline{Q}_{p,q}( (f_q\otimes 1) \overline{Q}_{q,n}(1)))
\leq c_{\mu}\|f\|_{v^{\alpha}}v(x)^{3\alpha}[1 + \sum_{q=p+1}^n \rho^{q-p}]
$$
from which we can conclude.
\end{proof}

\subsection{Controlling the Backward Part}

\begin{prop}\label{prop:backward}
Assume (A\ref{hyp:1}-\ref{hyp:5}). Then if $\|f\|_{v^{\alpha}}<+\infty$ for $\alpha\in(0,1/2)$ there exist a $c<+\infty$ 
 which depends only upon the constants in (A\ref{hyp:1}) and (A\ref{hyp:6}-\ref{hyp:5}),
 such that for any $x\in\mathsf{X}$, $p\geq 1$
\begin{equation}
\frac{h_{p,n}(x)}{\eta_p(Q_{p,n}(1))} \sum_{q=0}^{p-1}\eta_p(Q_{p,n}(1)[M_{p:q}(f_q)(x) - M_{p:q}(f_q)])
\leq c \|f\|_{v^{\alpha}} v(x)^{2\alpha}.
\label{eq:backward_master_ineq}
\end{equation}
\end{prop}

\begin{proof}
Consider the summand in \eqref{eq:backward_master_ineq}
$$
\eta_p(Q_{p,n}(1)[M_{p:q}(f_q)(x) - M_{p:q}(f_q)]) = \|f\|_{v^{\alpha}}
\eta_p\Big(\Big[Q_{p,n}(1)v(x)^{\alpha}v^{\alpha}\Big]\Big[\frac{[M_{p:q}(\frac{f_q}{\|f\|_{v^{\alpha}}})(x) - M_{p:q}(\frac{f_q}{\|f\|_{v^{\alpha}}})])}{v(x)^{\alpha}v^{\alpha}}
\Big]\Big).
$$
Then applying Lemma \ref{lem:vnorm_dob_backward}, we have the upper-bound
$$
\eta_p(Q_{p,n}(1)[M_{p:q}(f_q)(x) - M_{p:q}(f_q)]) \leq  c\|f\|_{v^{\alpha}} \rho^{(p-q-1)}
\eta_p(Q_{p,n}(1)v^{\alpha})v(x)^{\alpha}.
$$
Thus \eqref{eq:backward_master_ineq} is upper-bounded by
$$
c \|f\|_{v^{\alpha}} h_{p,n}(x)\eta_p(h_{p,n}v^{\alpha})v(x)^{\alpha}.
$$
Then we have 
\begin{eqnarray*}
c \|f\|_{v^{\alpha}} h_{p,n}(x)\eta_p(h_{p,n}v^{\alpha})v(x)^{\alpha}
& \leq & c \|f\|_{v^{\alpha}} [\sup_{n\geq 1}\sup_{0\leq p\leq n}\|h_{p,n}\|_{v^{\alpha}}]^2
v(x)^{\alpha}\eta_p(v^{2\alpha})v(x)^{\alpha}  \\ & \leq &
c \|f\|_{v^{\alpha}} [\sup_{n\geq 1}\sup_{0\leq p\leq n}\|h_{p,n}\|_{v^{\alpha}}]^2\sup_{p\geq 0}\|\eta_p\|_{v^{2\alpha}}
v(x)^{2\alpha}.
\end{eqnarray*}
By \cite[Propositions 1,2]{whiteley} $[\sup_{n\geq 1}\sup_{0\leq p\leq n}\|h_{p,n}\|_{v^{\alpha}}]^2\sup_{p\geq 0}\|\eta_p\|_{v^{2\alpha}}<+\infty$
and we conclude that
$$
\frac{h_{p,n}(x)}{\eta_p(Q_{p,n}(1))} \sum_{q=0}^{p-1}\eta_p(Q_{p,n}(1)[M_{p:q}(f_q)(x) - M_{p:q}(f_q)])
\leq c \|f\|_{v^{\alpha}} v(x)^{2\alpha}
$$
as was to be proven.
\end{proof}

\subsection{Technical Results}

\subsubsection{Forward Part}

\begin{lem}\label{lem:easy_part_forward}
Assume (A\ref{hyp:1}-\ref{hyp:3}) and (A\ref{hyp:2}-\ref{hyp:5}). Then for any $\alpha\in (0,1/2)$ there exist a $c<+\infty$ 
depending only in the constants in (A\ref{hyp:1}), (A\ref{hyp:6}-\ref{hyp:5}), such that for any
$n\geq 1$, $0\leq p < q <n$, $x\in\mathsf{X}$:
$$
h_{p,n}(x) \frac{Q_{p,q}(v^{\alpha} Q_{q,n}(1))(x)}{Q_{p,q}(Q_{q,n}(1))(x)} \frac{\eta_p[Q_{p,q}(v^{\alpha} Q_{q,n}(1))]}{\eta_p[Q_{p,q}(Q_{q,n}(1))]}
\leq c v(x)^{3\alpha}.
$$
\end{lem}

\begin{proof}
Note that throughout $c$ denotes a generic finite constant that may depend upon $\alpha$, but whose value may change upon each appearance.
Define the Markov semi-group $T_{p,q}(x,dy) = Q_{p,q}(x,dy)/Q_{p,q}(1)(x)$. Then we have
\begin{equation}
h_{p,n}(x) \frac{Q_{p,q}(v^{\alpha} Q_{q,n}(1))(x)}{Q_{p,q}(Q_{q,n}(1))(x)} \frac{\eta_p[Q_{p,q}(v^{\alpha} Q_{q,n}(1))]}{\eta_p[Q_{p,q}(Q_{q,n}(1))]} = 
h_{p,n}(x) \frac{T_{p,q}(v^{\alpha} h_{q,n})(x)}{T_{p,q}(h_{q,n})(x)}
\frac{\eta_p[h_{p,q}T_{p,q}(v^{\alpha} h_{q,n})]}{\eta_p[h_{p,q}T_{p,q}(h_{q,n})]}.
\label{eq:decomp_forward_easy}
\end{equation}

We will consider the R.H.S.~of \eqref{eq:decomp_forward_easy}; first the term:
$$
\frac{h_{p,n}(x)}{T_{p,q}(h_{q,n})(x)} = \frac{Q_{p,n}(1)(x)}{\prod_{s=p}^{n-1} \lambda_s} \frac{Q_{p,q}(1)(x) \prod_{s=q}^{n-1} \lambda_s}{Q_{p,n}(1)(x)}
$$
where $\lambda_s = \eta_s(G_s)$ and we have used, recursively, \cite[Lemma 1]{whiteley}. Then by cancelling, it clearly follows that
$$
\frac{h_{p,n}(x)}{T_{p,q}(h_{q,n})(x)} =h_{p,q}(x).
$$
Hence, combining our calculations together and returning to \eqref{eq:decomp_forward_easy},
we have established that
\begin{equation}
h_{p,n}(x) \frac{Q_{p,q}(v^{\alpha} Q_{q,n}(1))(x)}{Q_{p,q}(Q_{q,n}(1))(x)} \frac{\eta_p[Q_{p,q}(v^{\alpha} Q_{q,n}(1))]}{\eta_p[Q_{p,q}(Q_{q,n}(1))]}
=  h_{p,q}(x) T_{p,q}(v^{\alpha} h_{q,n})(x)\frac{\eta_p[h_{p,q}T_{p,q}(v^{\alpha} h_{q,n})]}{\eta_p[h_{p,q}T_{p,q}(h_{q,n})]}.
\label{eq:decomp_forward_easy1}
\end{equation}

We now focus on the term $1/\eta_p[h_{p,q}T_{p,q}(h_{q,n})]$ in \eqref{eq:decomp_forward_easy1}. We note that for any $x\in\mathsf{X}$:
$$
h_{p,q}(x)T_{p,q}(h_{q,n})(x) = \frac{Q_{p,q}(1)(x)}{\prod_{s=p}^{q-1}\lambda_s}\frac{Q_{p,n}(1)(x)}{Q_{p,q}(1)(x)\prod_{s=q}^{n-1}\lambda_s}
= h_{p,n}(x).
$$
By Lemma 10 of \cite{whiteley} for any arbitrary $d\in[\underline{d},\infty)$, $\inf_{n\geq 1}\inf_{0\leq p \leq n}\inf_{x_{\in C_d}} h_{p,n}(x) > 0$ and so for any $d$ as stated and by using the above calculation:
$$
\eta_p[h_{p,q}T_{p,q}(h_{q,n})] \geq \eta_p[\mathbb{I}_{C_d}h_{p,n}] \geq \eta_p(C_d)\bigg[\inf_{n\geq 1}\inf_{0\leq p \leq n}\inf_{x_{\in C_d}} h_{p,n}(x)\bigg].
$$
Now by using the proof of Lemma 8 of \cite{whiteley}, page 2527, we have for $d$ large enough, that there is a finite $c>0$ such that
$$
\inf_{p\geq 0} \eta_p(C_d)\bigg[\inf_{n\geq 1}\inf_{0\leq p \leq n}\inf_{x_{\in C_d}} h_{p,n}(x)\bigg] \geq c.
$$
Thus returning to \eqref{eq:decomp_forward_easy1}, we have
\begin{equation}
h_{p,n}(x) \frac{Q_{p,q}(v^{\alpha} Q_{q,n}(1))(x)}{Q_{p,q}(Q_{q,n}(1))(x)} \frac{\eta_p[Q_{p,q}(v^{\alpha} Q_{q,n}(1))]}{\eta_p[Q_{p,q}(Q_{q,n}(1))]}
\leq  c h_{p,q}(x) T_{p,q}(v^{\alpha} h_{q,n})(x) \eta_p[h_{p,q}T_{p,q}(v^{\alpha} h_{q,n})].
\label{eq:decomp_forward_easy2}
\end{equation}

Now using the above arguments, we have $\sup_{n\geq 1}\sup_{1\leq q \leq n}\|h_{q,n}\|_{v^{\alpha}} <+\infty$, so we have
for any $x\in\mathsf{X}$
$$
T_{p,q}(v^{\alpha} h_{q,n})(x) \leq c T_{p,q}(v^{2\alpha})(x)
$$
where $c$ does not depend upon $p,q,n$. Then using the calculations of \cite[Theorem 1]{whiteley}, which arrive at the equation (61), page 2532, one
has
\begin{equation}
T_{p,q}(v^{\alpha} h_{q,n})(x) \leq c \frac{v_{p,q,2\alpha}(x)}{\|h_{p,q}\|_{v^{2\alpha}}}
\label{eq:neat_trick}
\end{equation}
where $v_{p,q,2\alpha}(x) = v(x)^{2\alpha}\|h_{p,q}\|_{v^{2\alpha}}/h_{p,q}(x)$ and we are invoking Lemma 3 of \cite{whiteley}. Hence, returning to
\eqref{eq:decomp_forward_easy2}, we have
\begin{equation}
h_{p,n}(x) \frac{Q_{p,q}(v^{\alpha} Q_{q,n}(1))(x)}{Q_{p,q}(Q_{q,n}(1))(x)} \frac{\eta_p[Q_{p,q}(v^{\alpha} Q_{q,n}(1))]}{\eta_p[Q_{p,q}(Q_{q,n}(1))]}
\leq 
c v(x)^{3\alpha}
\eta_p[h_{p,q}T_{p,q}(v^{\alpha} h_{q,n})]
\label{eq:decomp_forward_easy3}
\end{equation}

We now turn to $\eta_p[h_{p,q}T_{p,q}(v^{\alpha} h_{q,n})]$ on the R.H.S.~of \eqref{eq:decomp_forward_easy3}. By using \eqref{eq:neat_trick},
we have
$$
\eta_p[h_{p,q}T_{p,q}(v^{\alpha} h_{q,n})] \leq c \eta_p(v^{2\alpha})
$$
where $c$ depends upon $\alpha$ only. Using Proposition 1 of \cite{whiteley} (noting again Lemma 3 of \cite{whiteley} and that $\alpha\in(0,1/2)$), we can thus conclude that:
$$
h_{p,n}(x) \frac{Q_{p,q}(v^{\alpha} Q_{q,n}(1))(x)}{Q_{p,q}(Q_{q,n}(1))(x)} \frac{\eta_p[Q_{p,q}(v^{\alpha} Q_{q,n}(1))]}{\eta_p[Q_{p,q}(Q_{q,n}(1))]}
\leq c v(x)^{3\alpha}
$$
which completes the proof.
\end{proof}

\begin{lem}\label{lem:forward_hard_part}
Assume (A\ref{hyp:1}-\ref{hyp:3}) and (A\ref{hyp:2}-\ref{hyp:5}). Then there exist a $d\in[\underline{d},\infty)$ such that 
for any $\alpha\in (0,1/3)$, $\beta\in(0,1)$  there exist a $0<c(d,\alpha,\beta)<+\infty$ such that for any, $n\geq 1$, $0\leq p < q <n$, $x\in\mathsf{X}$:
$$
\frac{h_{p,n}(x)\overline{\mathbb{E}}_{\delta_x\otimes\eta_p}\Big[\Big\{\prod_{s=p}^{q-1} \overline{G}_s(\overline{X}_s)\Big\}
\overline{v}(\overline{X}_q)^{\alpha} \overline{Q}_{q,n}(1)(\overline{X}_q)
\mathbb{I}_{ \overline{M}_{p,q}^{d} <\beta(q-p) }
\Big] }{Q_{p,n}(1)(x)\eta_p(Q_{p,n}(1))} \leq 
$$
$$
 c(d,\alpha,\beta) \mu(v^{3\alpha})v(x)^{3\alpha} \exp\{-(q-p)c(d,\alpha,\beta)]\}.
$$
\end{lem}

\begin{proof}
Throughout $c$ denotes a generic finite and positive constant that depends upon $\alpha,\beta,d$, but whose value may change upon each appearance.
The dependences of $c$ are omitted in the proof to simplity the notations.

We can rewrite
$$
\frac{h_{p,n}(x)\overline{\mathbb{E}}_{\delta_x\otimes\eta_p}\Big[\Big\{\prod_{s=p}^{q-1} \overline{G}_s(\overline{X}_s)\Big\}
\overline{v}(\overline{X}_q)^{\alpha} \overline{Q}_{q,n}(1)(\overline{X}_q)
\mathbb{I}_{ \overline{M}_{p,q}^{d} <\beta(q-p) }
\Big] }{Q_{p,n}(1)(x)\eta_p(Q_{p,n}(1))}
=
$$
\begin{equation}
\frac{h_{p,n}(x)\overline{\mathbb{E}}_{\delta_x\otimes\eta_p}\Big[\Big\{\prod_{s=p}^{q-1} \overline{G}_s(\overline{X}_s)\Big\}
\overline{v}(\overline{X}_q)^{\alpha} \overline{h}_{q,n}(\overline{X}_q)
\mathbb{I}_{ \overline{M}_{p,q}^{d} <\beta(q-p) }
\Big] }{Q_{p,q}(h_{q,n})(x)\eta_p(Q_{p,q}(h_{q,n}))}.
\label{eq:first_decomp}
\end{equation}

Now consider the term:
$
\frac{h_{p,n}(x)}{Q_{p,q}(h_{q,n})(x)}
$
in \eqref{eq:first_decomp}. We have
$$
\frac{h_{p,n}(x)}{Q_{p,q}(h_{q,n})(x)} = \frac{Q_{p,n}(1)(x)\prod_{s=q}^{n-1}\lambda_s}{\prod_{s=p}^{n-1}\lambda_s Q_{p,n}(1)(x)}= \frac{1}{\prod_{s=p}^{q-1}\lambda_s}.
$$
Now, using Propositions 1 and 2 of \cite{whiteley}, $\underline{\lambda} := \inf_{s\geq 0}\lambda_s >0$ and thus by the above calculation it follows that
$$
\frac{h_{p,n}(x)}{Q_{p,q}(h_{q,n})(x)} \leq \frac{1}{\underline{\lambda}^{q-p}}.
$$
This leaves us with 
$$
\frac{h_{p,n}(x)\overline{\mathbb{E}}_{\delta_x\otimes\eta_p}\Big[\Big\{\prod_{s=p}^{q-1} \overline{G}_s(\overline{X}_s)\Big\}
\overline{v}(\overline{X}_q)^{\alpha} \overline{Q}_{q,n}(1)(\overline{X}_q)
\mathbb{I}_{ \overline{M}_{p,q}^{d} <\beta(q-p) }
\Big] }{Q_{p,n}(1)(x)\eta_p(Q_{p,n}(1))}
=
$$
\begin{equation}
\frac{\overline{\mathbb{E}}_{\delta_x\otimes\eta_p}\Big[\Big\{\prod_{s=p}^{q-1} \overline{G}_s(\overline{X}_s)\Big\}
\overline{v}(\overline{X}_q)^{\alpha} \overline{h}_{q,n}(\overline{X}_q)
\mathbb{I}_{ \overline{M}_{p,q}^{d} <\beta(q-p) }
\Big] }{\underline{\lambda}^{q-p}\eta_p(Q_{p,q}(h_{q,n}))}.
\label{eq:second_decomp}
\end{equation}

The next term we consider on the R.H.S.~ of \eqref{eq:second_decomp} is $1/\eta_p(Q_{p,q}(h_{q,n}))$. Pick a $r\in[\underline{d},d)$ fixed. Then we have by repeatedly applying (A\ref{hyp:2})
\begin{eqnarray*}
\eta_p(Q_{p,q}(h_{q,n})) & \geq & \eta_p(Q_{p,q}(C_r)) \inf_{n\geq 1}\inf_{0\leq q \leq n}\inf_{x_{\in C_r}} h_{q,n}(x) \\
& \geq & \eta_p(C_r) (\epsilon_r^{-}\nu_r(C_r))^{q-p} \inf_{n\geq 1}\inf_{0\leq q \leq n}\inf_{x\in C_r} h_{q,n}(x).
\end{eqnarray*}
Now by Lemma 10 of \cite{whiteley} $\inf_{n\geq 1}\inf_{0\leq q \leq n}\inf_{x_{\in C_r}} h_{q,n}(x)>0$ and for $r$ and hence $d$ large enough
$\inf_{p\geq 0}\eta_p(C_r)>0$ by the proof of Lemma 8 page 2527 of \cite{whiteley}. Now fix $r$ from here-in. Thus we have shown that for $r, d$ large enough:
$$
\frac{h_{p,n}(x)\overline{\mathbb{E}}_{\delta_x\otimes\eta_p}\Big[\Big\{\prod_{s=p}^{q-1} \overline{G}_s(\overline{X}_s)\Big\}
\overline{v}(\overline{X}_q)^{\alpha} \overline{Q}_{q,n}(1)(\overline{X}_q)
\mathbb{I}_{ \overline{M}_{p,q}^{d} <\beta(q-p) }
\Big] }{Q_{p,n}(1)(x)\eta_p(Q_{p,n}(1))}
\leq
$$
\begin{equation}
c\frac{\overline{\mathbb{E}}_{\delta_x\otimes\eta_p}\Big[\Big\{\prod_{s=p}^{q-1} \overline{G}_s(\overline{X}_s)\Big\}
\overline{v}(\overline{X}_q)^{\alpha} \overline{h}_{q,n}(\overline{X}_q)
\mathbb{I}_{ \overline{M}_{p,q}^{d} <\beta(q-p) }
\Big] }{(\underline{\lambda}\epsilon_r^{-}\nu_r(C_r))^{q-p}}.
\label{eq:third_decomp}
\end{equation}

Now to complete the proof, we note that as $h_{q,n}\in\mathscr{L}_{v^{\alpha}}$ and $\sup_{n\geq 1}\sup_{0\leq q \leq n}\|h_{q,n}\|_{v^{\alpha}}<+\infty$, by Propositions 1, 2 and Lemma 3 of \cite{whiteley}, the upper-bound of 
the R.H.S.~of \eqref{eq:third_decomp}:
$$
c\frac{\overline{\mathbb{E}}_{\delta_x\otimes\eta_p}\Big[\Big\{\prod_{s=p}^{q-1} \overline{G}_s(\overline{X}_s)\Big\}
\overline{v}(\overline{X}_q)^{3\alpha} 
\mathbb{I}_{ \overline{M}_{p,q}^{d} <\beta(q-p) }
\Big] }{(\underline{\lambda}\epsilon_r^{-}\nu_r(C_r))^{q-p}}.
$$
Then by the proof of Theorem 1 of \cite{whiteley}, pages 2533-2534 we note
$$
\overline{\mathbb{E}}_{\delta_x\otimes\eta_p}\Big[\Big\{\prod_{s=p}^{q-1} \overline{G}_s(\overline{X}_s)\Big\}
\overline{v}(\overline{X}_q)^{3\alpha} 
\mathbb{I}_{ \overline{M}_{p,q}^{d} <\beta(q-p) }
\Big]
\leq c \mu(v^{3\alpha}) v(x)^{3\alpha} \exp\{-d\delta(q-p)(1-\beta)/2 + 3d\delta/2\}.
$$
Hence we have proved that for $r,d$ large enough
$$
\frac{h_{p,n}(x)\overline{\mathbb{E}}_{\delta_x\otimes\eta_p}\Big[\Big\{\prod_{s=p}^{q-1} \overline{G}_s(\overline{X}_s)\Big\}
\overline{v}(\overline{X}_q)^{\alpha} \overline{Q}_{q,n}(1)(\overline{X}_q)
\mathbb{I}_{ \overline{M}_{p,q}^{d} <\beta(q-p) }
\Big] }{Q_{p,n}(1)(x)\eta_p(Q_{p,n}(1))}
\leq
$$
$$
c \mu(v^{3\alpha}) v(x)^{3\alpha} \exp\{-(q-p)[d\delta (1-\beta)/2 + \log(\underline{\lambda}) + \log(\epsilon_r^{-1}\mu_r(C_r)))] + 3d\delta/2\}.
$$
On noting that $r$ is fixed, one can increase $d$ to ensure that the result holds true.
\end{proof}

\subsubsection{Backward Part}

\begin{lem}\label{lem:vnorm_dob_backward}
Assume (A\ref{hyp:1}-\ref{hyp:5}). Then for any $\alpha\in(0,1/2)$,
$p\geq 1$, $q\in\{0,\dots,p-1\}$  there exist a $c<+\infty$
which depends only upon the constants in (A\ref{hyp:1}),
(A\ref{hyp:6}-\ref{hyp:5}) such that
$$
\sup_{(x,z)\in\overline{\mathsf{X}}}\sup_{|f|\leq v^{\alpha}}
\frac{|M_{p:q}(f)(x) - M_{p:q}(f)(z)|}{\overline{v}(x,z)^{\alpha}}
\leq c\rho^{(p-q-1)}.
$$
\end{lem}

\begin{proof}
We start by using Lemma 4.3 of \cite{dds1}, which provides the neat reversal formula:
\begin{equation}
M_{p:q}(f)(x) = \frac{\eta_q(f Q_{q,p-1}[(Q_{p}(\cdot,x)])}{\eta_q(Q_{q,p-1}[Q_{p}(\cdot,x)])} \quad \forall x\in\mathsf{X}
\label{eq:back_rep}
\end{equation}
where we use the abuse of notation $\mu Q_p(\cdot,x) = \int \mu(dy) G_{p-1}(y) H_p(y,x)$ for any $\sigma-$finite measure $\mu$.

We first focus on the case that  $q\in\{0,\dots,p-2\}$. We note that using a similar proof to \cite[Lemma 1]{whiteley} that for
any $\varphi:\mathsf{X}\rightarrow\mathbb{R}$ 
\begin{equation}
\eta_{q}(Q_{q,p-1}(\varphi)) = \Big(\prod_{s=q}^{p-2}\lambda_s\Big) \eta_{p-1}(\varphi)
\label{eq:eigen_fn_fn}.
\end{equation} 
Using the representation \eqref{eq:back_rep} and the identity \eqref{eq:eigen_fn_fn}, we have that
$$
\frac{M_{p:q}(f)(x) - M_{p:q}(f)(z)}{\overline{v}(x,z)^{\alpha}}
= 
$$
\begin{equation}
\frac{(\eta_q\otimes\eta_q)(
f\{Q_{q,p-1}[Q_{p}(\cdot,x)]Q_{q,p-1}[Q_{p}(\cdot,z)]-
Q_{q,p-1}[Q_{p}(\cdot,z)]Q_{q,p-1}[Q_{p}(\cdot,x)]
\})}{\Big(\prod_{s=q}^{p-2}\lambda_s\Big)^2
\eta_{p-1}[Q_{p}(\cdot,x)]\eta_{p-1}[Q_{p}(\cdot,z)]\overline{v}(x,z)^{\alpha}}
\label{eq:backward_master_eq}.
\end{equation}

Consider the argument of the function
that is operated on by $(\eta_q\otimes\eta_q)$, when excluding $f$ on the R.H.S.~of \eqref{eq:backward_master_eq}.
This can be written as
$$
(\delta_s\otimes\delta_t - \delta_t\otimes\delta_s)(\overline{Q}_{q,p-1}(Q_p(\cdot,x)\otimes Q_p(\cdot,z))).
$$
Then by (A\ref{hyp:6}) as $Q_p(y,x)/\eta_{p-1}[Q_{p}(\cdot,x)]\in\mathscr{L}_{\overline{v}^{\alpha}}$, and via
decompositions and calculations in \cite{douc} and \cite{klep} (see e.g.~the proof of Theorem 1 of \cite{whiteley})
$$
\frac{(\delta_s\otimes\delta_t - \delta_t\otimes\delta_s)(\overline{Q}_{q,p-1}(Q_p(\cdot,x)\otimes Q_p(\cdot,z)))}{\eta_{p-1}[Q_{p}(\cdot,x)]\eta_{p-1}[Q_{p}(\cdot,z)]}
\leq c (\delta_s\otimes\delta_t) \overline{R}_{q,p-1}(\overline{v}^{\alpha}) ]\overline{v}(x,z)^{\alpha}
$$
where $c$ depends on $\sup_{p\geq 1}\|Q_p/\eta_{p-1}[Q_{p}]\|_{\overline{v}^{\alpha}}$ and
$$
\overline{R}_{r}(\bar{x},d\bar{y})=\overline{Q}_{r}(\bar{x},d\bar{y}) - \mathbb{I}_{\overline{C}_d}(\bar{x})(\epsilon_d^{-})^2\nu_d\otimes\nu_d(d\bar{y})
$$ 
with $\bar{x}=(x_1,x_2)\in\overline{\mathsf{X}}$, $\bar{y}=(y_1,y_2)\in\overline{\mathsf{X}}$
and $\overline{R}_{q,p-1}=\overline{R}_{q+1}\dots \overline{R}_{p-1}$.
By the calculations of \cite[Theorem 1, pp.~2532-2534]{whiteley}, we have that
$$
(\delta_s\otimes\delta_t) \overline{R}_{q,p-1}(\overline{v}^{\alpha})
\leq c\rho_d^{\beta(p-q-1)}\overline{Q}_{q,p-1}(\overline{v}^{\alpha})(s,t) +
c\exp\Big\{-(p-q-1)\Big[\frac{\delta d(1-\beta)}{2}-2b_{\underline{d}}\Big] + \frac{3\delta d}{2}\Big\}\overline{v}(s,t)^{\alpha}
$$
where $c$ does not depend upon $d$,  $d\geq \underline{d}$, $\beta\in(0,1)$ are arbitrary and $\rho_d=(1-\Big(\frac{\epsilon_d^{-}}{\epsilon_d^{+}}\Big)^2)$
Thus returning to \eqref{eq:backward_master_eq}, we have established that
$$
\frac{M_{p:q}(f)(x) - M_{p:q}(f)(z)}{\overline{v}(x,z)^{\alpha}}
\leq c \Big(\prod_{s=q}^{p-2}\lambda_s\Big)^{-2}\times
$$
\begin{equation}
(\eta_q\otimes\eta_q)\Big(v^{\alpha}\Big\{
\rho_d^{\beta(p-q-1)}\overline{Q}_{q,p-1}(\overline{v}^{\alpha}) +
\exp\Big\{-(p-q-1)\Big[\frac{\delta d(1-\beta)}{2}-2b_{\underline{d}}\Big] + \frac{3\delta d}{2}\Big\}\overline{v}^{\alpha}\Big\}
\Big)\label{eq:backward_dob_bd1}
\end{equation}
We split the R.H.S.~of \eqref{eq:backward_dob_bd1} into the sum of two expressions:
\begin{equation}
c\Big(\prod_{s=q}^{p-2}\lambda_s\Big)^{-2}
(\eta_q\otimes\eta_q)\Big(v^{\alpha}\rho_d^{\beta(p-q-1)}\overline{Q}_{q,p-1}(\overline{v}^{\alpha})\Big)
\label{eq:backward_dob_t1}
\end{equation}
and 
\begin{equation}
c\Big(\prod_{s=q}^{p-2}\lambda_s\Big)^{-2}
(\eta_q\otimes\eta_q)\Big(v^{\alpha}
\exp\Big\{-(p-q-1)\Big[\frac{\delta d(1-\beta)}{2}-2b_{\underline{d}}\Big] + \frac{3\delta d}{2}\Big\}\overline{v}^{\alpha}\Big)
\label{eq:backward_dob_t2}
\end{equation}

We start with \eqref{eq:backward_dob_t1}:
$$
c\rho_d^{\beta(p-q-1)} \frac{\eta_q(v^{\alpha}Q_{q,p-1}(v^{\alpha}))}{\prod_{s=q}^{p-2}\lambda_s}
\frac{\eta_q(Q_{q,p-1}(v^{\alpha}))}{\prod_{s=q}^{p-2}\lambda_s}
$$
By \cite[Theorem 1]{whiteley} we have the upper-bound
$$
c\rho_d^{\beta(p-q-1)} \eta_q(v^{\alpha}[h_{q,p-1}\eta_{p-1}(v^{\alpha}) + \tilde{\rho}^{\beta(p-q-1)}\mu(v^{\alpha})c_{\mu}v^{\alpha}])
\eta_q([h_{q,p-1}\eta_{p-1}(v^{\alpha}) + \tilde{\rho}^{\beta(p-q-1)}\mu(v^{\alpha})c_{\mu}v^{\alpha}])
$$
where $c<\infty$, $\tilde{\rho}\in(0,1)$ that does not depend on $d$.
As $\sup_{q\geq 1}\sup_{1\leq p \leq q+1}\|h_{q,p-1}\|_{v^{\alpha}}<+\infty$ by \cite[Proposition 2]{whiteley} 
and by Proposition 1 of \cite{whiteley} we have that $\sup_{p\geq 1} \|\eta_{p-1}(v^{\alpha})\|_{v^{\alpha}}<+\infty$
we have the upper-bound on 
\eqref{eq:backward_dob_t1}
$$
c\rho_d^{\beta(p-q-1)}\eta_q(v^{2\alpha})\eta_q(v^{\alpha})
$$
where again, $c$ does not depend on $d$. Noting that $\alpha\in(0,1/2)$ and applying Jensen and again \cite{whiteley} Proposition 1,  we have the upper-bound $c \rho_d^{\beta(p-q-1)}$ for $c$ independent of $d$.

Now, turning to \eqref{eq:backward_dob_t2}, by Proposition 2 of \cite{whiteley} $\inf_{p\geq 0} \lambda_p  = \underline{\lambda} > 0$, 
and, by the above argument $\sup_{p\geq 1} \|\eta_{p-1}(v^{2\alpha})\|_{v^{\alpha}}<+\infty$
hence we have the upper-bound on 
 \eqref{eq:backward_dob_t2}
$$
c\exp\Big\{-(p-q-1)\Big[\frac{\delta d(1-\beta)}{2}-2b_{\underline{d}} + 2\log(\underline{\lambda})\Big] + \frac{3\delta d}{2}\Big\}.
$$
Thus combining this upper-bound, with that of  $c \rho_d^{\beta(p-q-1)}$ on \eqref{eq:backward_dob_t1} and recalling that the sum of these terms upper-bounded
the L.H.S.~of \eqref{eq:backward_dob_bd1}, we have established that
$$
\frac{M_{p:q}(f)(x) - M_{p:q}(f)(z)}{v(x)^{\alpha}v(z)^{\alpha}} \leq c\Big[
\rho_d^{\beta(p-q-1)} + 
 \exp\Big\{-(p-q-1)\Big[\frac{\delta d(1-\beta)}{2}-2b_{\underline{d}} + 2\log(\underline{\lambda})\Big] + \frac{3\delta d}{2}\Big\}
\Big]
$$
where  $q\in\{0,\dots,p-2\}$, $c$ does not depend upon $d$ and $d>\underline{d}$ is arbitrary. As $d$ is arbitrary, we can conclude that for $d$ large enough,
there is a $\rho\in(0,1)$ such that for any $q\in\{0,\dots,p-2\}$
$$
\sup_{(x,z)\in\overline{\mathsf{X}}}\sup_{|f|\leq v^{\alpha}}
\frac{|M_{p:q}(f)(x) - M_{p:q}(f)(z)|}{v(x)^{\alpha}v(z)^{\alpha}}
\leq c\rho^{(p-q-1)}
$$
with $c<+\infty$.

For the case $q=p-1$ we have, by definition of the backward kernel
$$
\frac{M_{p,\eta_{p-1}}(f)(x) - M_{p,\eta_{p-1}}(f)(z)}{v(x)^{\alpha}v(z)^{\alpha}} = 
\frac{\eta_{p-1}(fQ_p(\cdot,x))}{\eta_{p-1}(Q_p(\cdot,x))v(x)^{\alpha}v(z)^{\alpha}} - 
\frac{\eta_{p-1}(fQ_p(\cdot,z))}{\eta_{p-1}(Q_p(\cdot,z))v(x)^{\alpha}v(z)^{\alpha}}.
$$
By (A\ref{hyp:6}) as $Q_p(y,x)/\eta_{p-1}[Q_{p}(\cdot,x)]\in\mathscr{L}_{\overline{v}^{\alpha}}$
and as $v\geq 1$, we have 
$$
\frac{M_{p,\eta_{p-1}}(f)(x) - M_{p,\eta_{p-1}}(f)(z)}{v(x)^{\alpha}v(z)^{\alpha}} \leq
c \eta_{p-1}(v^{2\alpha}).
$$
Using  $\alpha\in(0,1/2)$  and \cite[Proposition 1]{whiteley} we can conclude.
\end{proof}

\end{document}